\documentclass[11pt]{article}
\usepackage{amsmath}
\usepackage{mathrsfs}
\usepackage{amssymb}
\usepackage{amsthm}
\usepackage{graphicx}
\usepackage{enumitem}
\usepackage{subfigure}
\usepackage{color}
\usepackage{multirow}
\usepackage{verbatim}
\usepackage{algorithm}
\usepackage{algorithmic}
\usepackage{bbm}
\usepackage{bigstrut}
\usepackage[left=1in,top=1in,right=1in,bottom=1in,letterpaper]{geometry}
\usepackage[linktocpage,colorlinks,linkcolor=blue,anchorcolor=blue,citecolor=blue,urlcolor=blue]{hyperref}

\providecommand{\keywords}[1]
{\small	
  \textbf{\textit{Keywords---}} #1
}

\numberwithin{equation}{section}

\newtheorem{theorem}{Theorem}[section]

\newtheorem{lemma}[theorem]{Lemma}

\newtheorem{definition}[theorem]{Definition}
\newtheorem{fact}[theorem]{Fact}

\newtheorem{remark}[theorem]{Remark}

\newcommand{\M}{\mathcal{M}}
\newcommand{\U}{\mathcal{U}}

\newcommand{\A}{\mathcal{A}}

\newcommand{\G}{\mathcal{G}}
\newcommand{\J}{\mathcal{J}}
\newcommand{\GG}{\mathrm{G}}
\newcommand{\T}{\mathrm{T}}

\newcommand{\R}{\mathbb{R}}

\newcommand{\etal}{ et al. }
\newcommand{\argmin}{\mathop{\rm argmin}}
\newcommand{\br}{\mathbb{R}}

\newcommand{\half}{\frac{1}{2}}

\newcommand{\Tr}{\mathrm{Tr}}
\newcommand{\st}{\mathrm{s.t. }}

\newcommand{\vet}{\mathrm{vec}}
\newcommand{\St}{\mathrm{St}}

\newcommand{\be}{\begin{equation}}
\newcommand{\ee}{\end{equation}}
\newcommand{\ba}{\begin{array}}
\newcommand{\ea}{\end{array}}
\newcommand{\bad}{\begin{aligned}}
\newcommand{\ead}{\end{aligned}}

\newcommand{\normtwo}[1]{\| #1 \|}

\newcommand{\normfro}[1]{\| #1 \|_{\text{F}}}
\newcommand{\inp}[2]{\langle #1, #2 \rangle}

\newcommand{\Retr}{\mathrm{Retr}}
\newcommand{\Id}{\mathrm{Id}}

\newcommand{\LCal}{\mathcal{L}}
\newcommand{\grad}{\mathrm{grad}}
\newcommand{\vvec}{\mathrm{vec}}
\newcommand{\svec}{\overline{\mathrm{vec}}}
\newcommand{\prox}{\mathrm{prox}}
\newcommand{\Proj}{\mathrm{Proj}}

\begin{document}
\title{{Proximal Gradient Method for Nonsmooth Optimization \\ over the Stiefel Manifold}}
\author{Shixiang Chen\thanks{Department of Systems Engineering and Engineering Management, The Chinese University of Hong Kong}
\and Shiqian Ma\thanks{Department of Mathematics, University of California, Davis}
\and Anthony Man-Cho So\thanks{Department of Systems Engineering and Engineering Management, and, by courtesy, CUHK-BGI Innovation
Institute of Trans-omics, The Chinese University of Hong Kong}
\and Tong Zhang\thanks{Departments of Computer Science and Mathematics, The Hong Kong University of Science and Technology}}
\date{May 05, 2019}
\maketitle

\begin{abstract}
 {We consider optimization problems over the Stiefel manifold whose objective function is the summation of a smooth function and a nonsmooth function.} Existing methods for solving this kind of problems can be classified into three classes. Algorithms in the first class rely on information of the subgradients of the objective function and thus tend to converge slowly in practice. Algorithms in the second class are proximal point algorithms, which involve subproblems that can be as difficult as the original problem. Algorithms in the third class are based on operator-splitting techniques, but they usually lack rigorous convergence guarantees. In this paper, we propose a retraction-based proximal gradient method for solving this class of problems. We prove that the proposed method globally converges to a stationary point. Iteration complexity for obtaining an $\epsilon$-stationary solution is also analyzed. Numerical results on solving sparse PCA and compressed modes problems are reported to demonstrate the advantages of the proposed method.
\end{abstract}


\section{Introduction}

Optimization over Riemannian manifolds has recently drawn a lot of attention due to its applications in many different fields, including low-rank matrix completion \cite{RTRMC-2011,Vandereycken-matrix-completion-2013}, phase retrieval \cite{Boumal-phase-retrieval-2018,Sun-Ju-geometric-phase-retrieval-2018}, phase synchronization \cite{Boumal-phase-synchronization-2016,Liu-generalized-power-phase-synchronization-2017}, blind deconvolution \cite{Huang-2018}, and dictionary learning \cite{Sra-Riemannian-dictionary-learning-2016,Sun-dictionary-recovery-sphere-2017}. Manifold optimization seeks to minimize an objective function over a smooth manifold. Some commonly encountered manifolds include the sphere, Stiefel manifold, Grassmann manifold, and Hadamard manifold. The recent monograph by Absil\etal\cite{Absil2009} studies this topic in depth. In particular, it studies several important classes of algorithms for manifold optimization with smooth objective, including line-search method, Newton's method, and trust-region method. There are also many gradient-based algorithms for solving manifold optimization problems, including \cite{Wen-Yin-2013,Shamir-svd-2015,Shamir-svd-2016,Liu-So-Wu-2018,Jiang-svrg-manifold-2017,zhang2016fast}.
However, all these methods require computing the derivatives of the objective function and do not apply to the case where the objective function is nonsmooth. 

 {
In this paper, we focus on a class of nonsmooth nonconvex optimization problems over the Stiefel manifold that takes the form
\be\label{prob-g}
\min \ F(X):=f(X)+h(X), \ \st, X\in \M:=\St(n,r)=\{X: X\in\R^{n\times r}, X^\top X=I_r\},
\ee
where $I_r$ denotes the $r\times r$ identity matrix ($r\leq n$). Throughout this paper, we make the following assumptions about \eqref{prob-g}:
\begin{enumerate}[label=(\roman*)]
\item $f$ is smooth, possibly nonconvex, and its gradient $\nabla f$ is Lipschitz continuous with Lipschitz constant $L$.
\item $h$ is convex, possibly nonsmooth, and is Lipschitz continuous with constant $L_h$. Moreover, the proximal mapping of $h$ is easy to find.
\end{enumerate}
Note that here the smoothness, Lipschitz continuity and convexity are interpreted when the function in question is considered as a function in the ambient Euclidean space.
}

We restrict our discussions in this paper to \eqref{prob-g} because it already finds many important applications in practice. 
In the following we briefly mention some representative applications of \eqref{prob-g}. For more examples of manifold optimization with nonsmooth objectives, we refer the reader to \cite{Absil-nonsmooth-examples}.

{\bf Example 1. Sparse Principal Component Analysis.} Principal Component Analysis (PCA), proposed by Pearson \cite{pearson1901liii} and later developed by Hotelling \cite{hotelling1933analysis}, is one of the most fundamental statistical tools in analyzing high-dimensional data. Sparse PCA seeks principal components with very few nonzero components. For given data matrix $A\in\br^{m\times n}$, the sparse PCA that seeks the leading $r$ $(r<\min\{m,n\})$ sparse loading vectors can be formulated as
\be\label{spca}
\ba{ll}
\min_{X\in\br^{n\times r}} & -\Tr(X^\top A^\top A X) + \mu\|X\|_1 \\
\st  & X^\top X = I_r,
\ea
\ee
where $\Tr(Y)$ denotes the trace of matrix $Y$, the $\ell_1$ norm is defined as $\|X\|_1=\sum_{ij}|X_{ij}|$, $\mu>0$ is a weighting parameter. 
{This is the original formulation of sparse PCA as proposed by Jolliffe\etal in \cite{Jolliffe2003}, where the model is called SCoTLASS and imposes sparsity and orthogonality to the loading vectors simultaneously. When $\mu=0$, \eqref{spca} reduces to computing the leading $r$ eigenvalues and the corresponding eigenvectors of $A^\top A$. When $\mu>0$, the $\ell_1$ norm $\|X\|_1$ can promote sparsity of the loading vectors. There are many numerical algorithms for solving \eqref{spca} when $r=1$. In this case, \eqref{spca} is relatively easy to solve because $X$ reduces to a vector and the constraint set reduces to a sphere. However, there has been very limited literature for the case $r>1$. Existing works, including \cite{Zou-spca-2006,daspremont-sparsePCA-direct-formulation-2007,Shen-Huang-spca-2008,Journee-Nesterov-sparsePCA-JMLR-2010,Ma-SPCA-2011-submit}, do not impose orthogonal loading directions. As discussed in \cite{Journee-Nesterov-sparsePCA-JMLR-2010}, ``Simultaneously
enforcing sparsity and orthogonality seems to be a hard (and perhaps questionable) task.'' We refer the interested reader to \cite{Zou-Xue-spca-survey-2018} for more details on existing algorithms for solving sparse PCA. As we will discuss later, our algorithm can solve \eqref{spca} with $r>1$ (i.e., imposing sparsity and orthogonality simultaneously) efficiently.}

{\bf Example 2. Compressed Modes in Physics.} This problem seeks spatially localized (``sparse'') solutions of the independent-particle Schr\"odinger's equation. Sparsity is achieved by adding an $L_1$ regularization of the wave functions, which leads to solutions with compact support (``compressed modes''). For 1D free-electron case, after proper discretization, this problem can be formulated as
\be\label{CM}
\ba{ll}
\min_{X\in\br^{n\times r}} & \Tr(X^\top H X) + \mu\|X\|_1 \\
\st  & X^\top X = I_r,
\ea
\ee
where $H$ denotes the discretized Schr\"odinger operator. Note that the $L_1$ regularization reduces to the $\ell_1$ norm of $X$ after discretization. We refer the reader to \cite{Ozolins2013} for more details of this problem. Note that \eqref{spca} and \eqref{CM} are different in the way that $H$ and $A^\top A$ have totally different structures. {In particular, $H$ is the discretized Schr\"odinger Hamiltonian, which is a block circulant matrix, while $A$ in \eqref{spca} usually comes from statistical data and thus $A^\top A$ is usually dense and unstructured.} These differences may affect the performance of algorithms for solving them.

{\bf Example 3. Unsupervised Feature Selection.} It is much more difficult to select the discriminative features  {in unsupervised learning than supervised learning}. There are some recent works that model this task as a manifold optimization problem in the form of \eqref{prob-g}. For instance, \cite{Yang2011} and \cite{Tang2012} assume that there is a linear classifier $W$ which classifies each data point $x_i$ (where $i=1,\ldots,n$) in the training data set to a class, and by denoting $G_i=W^\top x_i$, $[G_1,\ldots,G_n]$ gives a scaled label matrix which can be used to define some local discriminative scores. The target is to train a $W$ such that the local discriminative scores are the highest for all the training data $x_1,\ldots,x_n$. It is suggested in \cite{Yang2011} and \cite{Tang2012} to solve the following model to find $W$:
\[ 
\ba{ll}
\min_{W\in\br^{n\times r}} & \Tr(W^\top MW) + \mu\|W\|_{2,1} \\
\st  & W^\top W = I_r,
\ea
\] 
where $M$ is a given matrix computed from the input data, the $\ell_{2,1}$ norm is defined as $\|W\|_{2,1}=\sum_{i=1}^n\|W(i,:)\|_2$ with $W(i,:)$ being the $i$-th row of $W$, which promotes the row sparsity of $W$, and the orthogonal constraint is imposed to avoid arbitrary scaling and the trivial solution of all zeros. We refer the reader to \cite{Yang2011} and \cite{Tang2012} for more details.


{\bf Example 4. Sparse Blind Deconvolution.} Given the observations
\[y = a_0\circledast x_0\in\br^m,\]
how can one recover both the convolution kernel $a_0\in\br^k$ and signal $x_0\in\br^m$? Here $x_0$ is assumed to have a sparse and random support and $\circledast$ denotes the convolution operator. This problem is known as sparse blind deconvolution. Some recent works on this topic suggest the following optimization formulation to recover $a_0$ and sparse $x_0$ (see, e.g., \cite{Zhang-cvpr-2017}):
\[ 
\ba{ll}
\min_{a,x} & \|y - a\circledast x\|_2^2 + \mu\|x\|_1 \\
\st        & \|a\|_2 = 1.
\ea
\] 
Note that the sphere constraint here is a special case of the Stiefel manifold; i.e., $\St(k,1)$.

{\bf Example 5. Nonconvex Regularizer.} Problem \eqref{prob-g} also allows nonconvex regularizer functions. For example, instead of using the $\ell_1$ norm to promote sparsity, we can use the MCP (minimax concave penalty) function \cite{Zhang-MCP-2010}, which has been widely used in statistics. The MCP function is nonconvex and is given by
\[ 
P(x) = \left\{\ba{ll} \lambda |x| - \frac{x^2}{2\lambda}, & \mbox{ if } |x|\leq\gamma\lambda, \\
                                \frac{1}{2}\gamma\lambda^2, & \mbox{ otherwise}, \ea\right.
\] 
where $\lambda$ and $\gamma$ are given parameters, and $x\in\br$. If we replace the $\ell_1$ norm in sparse PCA \eqref{spca} by MCP, it reduces to
\be\label{spca-MCP}
\ba{ll}
\min_{X\in\br^{n\times r}} & -\Tr(X^\top A^\top A X) + \mu\sum_{ij}P(X_{ij}) \\
\st  & X^\top X = I_r.
\ea
\ee
It is easy to see that the objective function in \eqref{spca-MCP} can be rewritten as $f_1(X)+f_2(X)$, with $f_1(X) = -\Tr(X^\top A^\top A X) + \mu(\sum_{ij}P(X_{ij}) - \lambda\|X\|_1)$ and $f_2(X) = \mu\lambda\|X\|_1$. Note that $f_1$ is smooth and its gradient is Lipschitz continuous. Therefore, \eqref{spca-MCP} is an instance of \eqref{prob-g}.

{\bf Our Contributions.} Due to the needs of the above-mentioned applications, it is highly desirable to design an efficient algorithm for solving \eqref{prob-g}. In this paper, we propose a proximal gradient method for solving it. The proposed method, named ManPG (Manifold Proximal Gradient Method), is based on the proximal gradient method with a retraction operation to keep the iterates feasible with respect to the manifold constraint. Each step of ManPG involves solving a well-structured convex optimization problem, which can be done efficiently by the semi-smooth Newton method. We prove that ManPG converges to a stationary point of \eqref{prob-g} globally. We also analyze the iteration complexity of ManPG for obtaining an $\epsilon$-stationary point. Numerical results on sparse PCA \eqref{spca} and compressed modes \eqref{CM} problems show that our ManPG algorithm compares favorably with existing methods. 

{\bf Notation.} The following notation is adopted throughout this paper. The tangent space to $\M$ at point $X$ is denoted by $\T_X \M$. We use $\inp{A}{B}=\Tr(A^\top B)$ to denote the Euclidean inner product of two matrices $A,B$.  {We consider the Riemannian metric on $\M$ that is induced from the Euclidean inner product; i.e, for any $\xi,\eta \in \T_X\M$, we have $\inp{\xi}{\eta}_X = \Tr(\xi^\top\eta)$.} 
We use $\normfro{X}$ to denote the Frobenius norm of $X$ and $\normtwo{\A}_{op}$ to denote the operator norm of a linear operator $\A$. The Euclidean gradient of a smooth function $f$ is denoted as $\nabla f$ and the Riemannian gradient of $f$ is denoted as $\grad\,f$. Note that by our choice of the Riemannian metric, we have $\grad\,f(X)=\Proj_{\T_X\M}\nabla f(X)$, the orthogonal projection of $\nabla f(X)$ onto the tangent space.
 {According to \cite{Absil2009}, the projection of $Y$ onto the tangent space at $X\in \St(n,r)$ is given by $\Proj_{T_{X}\St(n,r)}=(I_n-XX^\top) Y+\half X(X^\top Y - Y^\top X)$}. We use $\Retr$ to denote the retraction operation. {For a convex function $h$, its Euclidean subgradient and Riemannian subgradient are denoted by $\partial h$ and $\hat{\partial}h$, respectively.} We use $\vvec(X)$ to denote the vector formed by stacking the column vectors of $X$. The set of $r\times r$ symmetric matrices is denoted by $S^r$. Given an $X\in S^r$, we use $\svec(X)$ to denote the $\frac{1}{2}r(r+1)$-dimensional vector obtained from $\vvec(X)$ by eliminating all super-diagonal elements of $X$. We denote $Z\succeq 0$ if $(Z+Z^\top)/2$ is positive semidefinite.
The proximal mapping of $h$ at point $X$ is defined by $\prox_{h}(X)=\argmin_{Y}\frac{1}{2}\normfro{Y-X}^2+h(Y)$.

{\bf Organization.} The rest of this paper is organized as follows. In Section \ref{sec:review} we briefly review existing works on solving manifold optimization problems with nonsmooth objective functions. We introduce some preliminaries of manifolds in Section \ref{sec:preliminary}. The main algorithm ManPG and the semi-smooth Newton method for solving the subproblem are presented in Section \ref{sec:ManPG}. In Section \ref{sec:convergence}, we establish the global convergence of ManPG and analyze its iteration complexity for obtaining an $\epsilon$-stationary solution. Numerical results of ManPG on solving compressed modes problems in physics and sparse PCA are reported in Section \ref{sec:numerical}. Finally, we draw some concluding remarks in Section \ref{sec:conclude}.

\section{Nonsmooth Optimization over Riemannian Manifold}\label{sec:review}


Unlike manifold optimization with a smooth objective, which has been studied extensively in the monograph \cite{Absil2009}, the literature on manifold optimization with a nonsmooth objective has been relatively limited. Numerical algorithms for solving manifold optimization with nonsmooth objectives can be roughly classified into three categories: subgradient-oriented methods, proximal point algorithms, and operator-splitting methods. We now briefly discuss the existing works in these three categories.

\subsection{Subgradient-oriented Methods}
Algorithms in the first category include the ones proposed in \cite{ferreira1998subgradient,Borckmans-2014,grohs2016varepsilon,Hosseini-2017-KL,Hosseini-Uschmajew-2017,hosseini2018line,Bacak-2016,Dirr-2006,Grohs-nonsmooth-trust-region-2016}, {which are all subgradient-oriented methods.}  {Ferreira and Oliveria \cite{ferreira1998subgradient} studied the convergence of subgradient method for minimizing a convex function over a Riemannian manifold. The subgradient method generates the iterates via
\[ 
X_{k+1} = \text{exp}_{X_k}(t_k V_k),
\] 
where $\text{exp}_{X_k}$ is the exponential mapping at $X_k$ and $V_k$ denotes a Riemannian subgradient of the objective. Like the subgradient method in Euclidean space, the stepsize $t_k$ is chosen to be diminishing to guarantee convergence. However, the result in \cite{ferreira1998subgradient} does not apply to \eqref{prob-g} because it is known that every smooth function that is convex on a compact Riemannian manifold is a constant \cite{bishop1969manifolds}.} This motivated some more advanced works on Riemannian subgradient method. Specifically, Dirr\etal\cite{Dirr-2006} and Borckmans\etal\cite{Borckmans-2014} proposed subgradient methods on manifold  {for the case where the objective function is the pointwise maximum of smooth functions. In this case, some generalized gradient can be computed and a descent direction can be found by solving a quadratic program.} Grohs and Hosseini \cite{grohs2016varepsilon} proposed a Riemannian $\varepsilon$-subgradient method. Hosseini and Uschmajew \cite{Hosseini-Uschmajew-2017} proposed a Riemannian gradient sampling algorithm.  {Hosseini\etal\cite{hosseini2018line} generalized the Wolfe conditions and extended the BFGS algorithm to nonsmooth functions on Riemannian manifolds.} Grohs and Hosseini \cite{Grohs-nonsmooth-trust-region-2016} generalized a nonsmooth trust region method to manifold optimization. Hosseini \cite{Hosseini-2017-KL} studied the convergence of some subgradient-oriented descent methods based on the Kurdyka-{\L}ojasiewicz (K{\L}) inequality.  {Roughly speaking, all the methods studied in \cite{Dirr-2006,Borckmans-2014,grohs2016varepsilon,Hosseini-Uschmajew-2017,hosseini2018line,Grohs-nonsmooth-trust-region-2016,Hosseini-2017-KL} require subgradient information to build a quadratic program to find a descent direction:
\be\label{subproblem:subgrad}
\hat{g}\longleftarrow\min_{g\in \text{conv}(W)} \normtwo{g},
\ee
where $\text{conv}(W)$ denotes the convex hull of set $W=\{G_j,j=1,\ldots,J\}$, $G_j$ is {the Riemannian gradient of a differentiable point around} the current iterate $X$, and $J$ usually needs to be larger than the dimension of $\M$. Subsequently, the iterate $X$ is updated by $X^+= \Retr_{X}(\alpha \hat{g})$, where the stepsize $\alpha$ is found by line search. For high-dimensional problems on the Stiefel manifold $\St(n,r)$, \eqref{subproblem:subgrad} can be difficult to solve because $n$ is large.} Since subgradient algorithm is known to be slower than the gradient algorithm and proximal gradient algorithm in Euclidean space, it is expected that these subgradient-based algorithms are not as efficient as gradient algorithms and proximal gradient algorithms on manifold in practice.

\subsection{Proximal Point Algorithms}
Proximal point algorithms (PPAs) for solving manifold optimization are also studied in the literature. Ferreira and Oliveira \cite{Ferreira-Oliveira-PPA-Manifold-2002} extended PPA to manifold optimization, which in each iteration needs to minimize the original function plus a proximal term over the manifold. However, there are two issues that limit its applicability. The first is that the subproblem can be as difficult as the original problem. For example, Bacak\etal\cite{Bacak-2016} suggested to use the subgradient method to solve the subproblem, {but they require the subproblem to be in the form of the pointwise maximum of smooth functions tackled in \cite{Borckmans-2014}.} 
The second is that the discussions in the literature mainly focus on the Hadamard manifold and exploit heavily the convexity assumption of the objective function. Thus, they do not apply to compact manifolds such as $\St(n,r)$. Bento\etal\cite{Bento-PPA-manifold-2016} aimed to resolve the second issue and proved the convergence of the PPA for more general Riemannian manifolds under the assumption that the K{\L} inequality holds for the objective function. In \cite{Bento-convergence-inexact-descent-2011}, Bento\etal analyzed the convergence of some inexact descent methods based on the K{\L} inequality, including the PPA and steepest descent method. In a more recent work~\cite{Bento-iteration-complexity-2017}, Bento\etal studied the iteration complexity of PPA under the assumption that the constraint set is the Hadamard manifold and the objective function is convex. 
Nevertheless, the results in \cite{Ferreira-Oliveira-PPA-Manifold-2002,Bento-PPA-manifold-2016,Bento-convergence-inexact-descent-2011,Bento-iteration-complexity-2017} seem to be of theoretical interest only because no numerical results were shown. As mentioned earlier, this could be due to the difficulty in solving the PPA subproblems. 

\subsection{Operator Splitting Methods}
Operator-splitting methods do not require subgradient information, and existing works in the literature mainly focus on the Stiefel manifold. Note that \eqref{prob-g} is challenging because of the combination of two difficult terms: Riemannian manifold and nonsmooth objective. If only one of them is present, then the problem is relatively easy to solve. Therefore, the alternating direction method of multipliers (ADMM) becomes a natural choice for solving \eqref{prob-g}. ADMM for solving convex optimization problems with two block variables is closely related to the famous Douglas-Rachford operator splitting method, which has a long history \cite{Glowinski-Marrocco-1975,Gabay-Mercier-1976,Lions-Mercier-79,Fortin-Glowinski-1983,Glowinski-LeTallec-89,Eckstein-thesis-89}. The renaissance of ADMM was initiated by several papers around 2007-2008, where it was successfully applied to solve various signal processing~\cite{Combettes-Pesquet-DR-2007} and image processing problems \cite{Yang-Yin-Zhang-Wang-08,Goldstein-Osher-2008,Afonso-BD-Figueiredo-2009}.
The recent survey paper \cite{Boyd-etal-ADM-survey-2011} popularized this method in many areas. Recently, there have been some emerging interests in ADMM for solving manifold optimization of the form \eqref{prob-g}; see, e.g., \cite{Lai-Osher-soc-2014,Kovnatsky2016,Zhang-Ma-Zhang-manifold-2017,WangYinZeng15}. However, the algorithms presented in these papers either lack convergence guarantee (\cite{Lai-Osher-soc-2014,Kovnatsky2016}) or their convergence needs further conditions that do not apply to \eqref{prob-g} (\cite{WangYinZeng15,Zhang-Ma-Zhang-manifold-2017}).


Here we briefly describe the SOC method (Splitting method for Orthogonality Constrained problems) presented in \cite{Lai-Osher-soc-2014}.
The SOC method aims to solve
\[ 
\min \ J(X), \ \st, \ X \in\M
\] 
by introducing an auxiliary variable $P$ and considering the following reformulation:
\be\label{soc-prob-reform}
\min \ J(P), \ \st, P = X, X \in\M.
\ee
By associating a Lagrange multiplier $\Lambda$ to the linear equality constraint, the augmented Lagrangian function of \eqref{soc-prob-reform} can be written as
\[\LCal_\beta(X,P;\Lambda):= J(P)-\langle \Lambda,P-X\rangle + \frac{\beta}{2}\|P-X\|_F^2,\]
where $\beta>0$ is a penalty parameter.
The SOC algorithm then generates its iterates as follows:
\[ 
\ba{lll}
P^{k+1} & := & \argmin_P \ \LCal_\beta(P,X^k;\Lambda^k), \\
X^{k+1} & := & \argmin_X \ \LCal_\beta(P^{k+1},X;\Lambda^k), \st, X\in\M,\\
\Lambda^{k+1} & := & \Lambda^k - \beta(P-X).
\ea
\] 
Note that the $X$-subproblem corresponds to the projection onto $\M$, and the $P$-subproblem is an unconstrained problem whose complexity depends on the structure of $J$. In particular, if $J$ is smooth, then the $P$-subproblem can be solved iteratively by the gradient method; if $J$ is nonsmooth and has an easily computable proximal mapping, then the $P$-subproblem can be solved directly by computing the proximal mapping of $J$.

The MADMM (manifold ADMM) algorithm presented in \cite{Kovnatsky2016} aims to solve the following problem:
\be\label{madmm-reform}
\min_{X,Z} \ f(X)+g(Z), \ \st, \ Z = AX, \ X\in \St(n,r),
\ee
where $f$ is smooth and $g$ is nonsmooth with an easily computable proximal mapping. The augmented Lagrangian function of \eqref{madmm-reform} is
\[\LCal_\beta(X,Z;\Lambda):=f(X)+g(Z)-\langle\Lambda,Z-AX\rangle+\frac{\beta}{2}\|Z-AX\|_F^2 \]
and the MADMM algorithm generates its iterates as follows:
\[ 
\ba{lll}
X^{k+1} & := & \argmin_X \ \LCal_\beta(X,Z^{k};\Lambda^{k}), \ \st, X\in \St(n,r), \\
Z^{k+1} & := & \argmin_Z \ \LCal_\beta(X^{k+1},Z;\Lambda^{k}), \\
\Lambda^{k+1} & := & \Lambda^{k} - \beta(Z^{k+1}-AX^{k+1}).
\ea
\] 
Note that the $X$-subproblem is a smooth optimization problem on the Stiefel manifold, and the authors suggested to use the Manopt toolbox \cite{Boumal2014} to solve it. The $Z$-subproblem corresponds to the proximal mapping of function $g$.

As far as we know, however, the convergence guarantees of SOC and MADMM are still missing in the literature. Though there are some recent works that analyze the convergence of ADMM for nonconvex problems \cite{WangYinZeng15,Zhang-Ma-Zhang-manifold-2017}, their results need further conditions that do not apply to \eqref{prob-g} and its reformulations \eqref{soc-prob-reform} and \eqref{madmm-reform}.

More recently, some other variants of the augmented Lagrangian method are proposed to deal with \eqref{prob-g}. In \cite{Chen2016}, Chen\etal proposed a PAMAL method which hybridizes an augmented Lagrangian method with the proximal alternating minimization method \cite{attouch2010proximal}. More specifically, PAMAL solves the following reformulation of \eqref{prob-g}:
\be\label{pamal-reform}
\min_{X,Q,P} \ f(P) + h(Q), \ \st,\ Q = X, P = X, X \in \St(n,r).
\ee
By associating Lagrange multipliers $\Lambda_1$ and $\Lambda_2$ to the two linear equality constraints, the augmented Lagrangian function of \eqref{pamal-reform} can be written as
\[\LCal_\beta(X,Q,P;\Lambda_1,\Lambda_2):=f(P)+h(Q)-\langle\Lambda_1,Q-X\rangle-\langle\Lambda_2,P-X\rangle+\frac{\beta}{2}\|Q-X\|_F^2+\frac{\beta}{2}\|P-X\|_F^2,\]
where $\beta>0$ is a penalty parameter. The augmented Lagrangian method for solving \eqref{pamal-reform} is then given by
\be\label{pamal}
\ba{rll}
(X^{k+1},Q^{k+1},P^{k+1}) & := & \argmin_{X,Q,P} \ \LCal_\beta(X,Q,P;\Lambda_1^{k},\Lambda_2^{k}), \ \st, \ X\in \St(n,r), \\
\Lambda_1^{k+1} & := & \Lambda_1^{k} - \beta(Q^{k+1}-X^{k+1}), \\
\Lambda_2^{k+1} & := & \Lambda_2^{k} - \beta(P^{k+1}-X^{k+1}).
\ea
\ee
Note that the subproblem in \eqref{pamal} is still difficult to solve. Therefore, the authors of \cite{Chen2016} suggested to use the proximal alternating minimization method \cite{attouch2010proximal} to solve the subproblem in \eqref{pamal} inexactly. They named the augmented Lagrangian method \eqref{pamal} with subproblems being solved by the proximal alternating minimization method as PAMAL. They proved that under certain conditions, any limit point of the sequence generated by PAMAL is a KKT point of \eqref{pamal-reform}. It needs to be pointed out that the proximal alternating minimization procedure involves many parameters that need to be tuned in order to solve the subproblem inexactly. Our numerical results in Section \ref{sec:numerical} indicate that the performance of PAMAL significantly depends on the setting of these parameters.

In \cite{Zhu2017}, Zhu\etal studied another algorithm called EPALMAL for solving \eqref{prob-g} that is based on the augmented Lagrangian method and the PALM algorithm \cite{Bolte-Sabach-Teboulle-2014}. The difference between EPALMAL and PAMAL is that they use different algorithms to minimize the augmented Lagrangian function inexactly. In particular, EPALMAL uses the PALM algorithm \cite{Bolte-Sabach-Teboulle-2014}, while PAMAL uses PAM \cite{attouch2010proximal}. It is also shown in \cite{Zhu2017} that any limit point of the sequence generated by EPALMAL is a KKT point. However, their result assumes that the iterate sequence is bounded, which holds automatically if the manifold in question is bounded but is hard to verify otherwise.

\section{Preliminaries on Manifold Optimization }\label{sec:preliminary}

{We first introduce the elements of manifold optimization that will be needed in the study of~\eqref{prob-g}. In fact, our discussion in this section applies to the case where $\M$ is any embedded submanifold of an Euclidean space.
To begin, we say that a function $F$ is locally Lipschitz continuous if for any $X\in\M$, it is Lipschitz continuous in a neighborhood of $X$. Note that if $F$ is locally Lipschitz continuous in the Euclidean space $\mathcal{E}$, then it is also locally Lipschitz continuous when restricted to the embedded submanifold $\M$ of $\mathcal{E}$.

\begin{definition}(Generalized Clarke subdifferential \cite{hosseini2011generalized})
For a locally Lipschitz function $F$ on $\M$, the Riemannian generalized directional derivative of $F$ at $X\in\M$ in the direction $V$ is defined by
\[ 
F^{\circ}(X,V) =\limsup\limits_{Y\rightarrow X,t\downarrow 0}\frac{F\circ \phi^{-1}(\phi(Y)+tD\phi(X )[V])-f\circ \phi^{-1}(\phi(Y))}{t},
\] 
where $(\phi,U)$ is a coordinate chart at $X$.
The generalized gradient or the Clarke subdifferential of $F$ at $X\in\M$, denoted by $\hat{\partial} F(X)$, is given by
\[ 
\hat{\partial} F(X)=\{\xi\in \T_X\M :\inp{\xi}{V}\leq F^{\circ}(X,V), \ \forall V\in \T_X\M \}.
\] 
\end{definition}
\begin{definition}(\cite{Yang-manifold-optimality-2014})
A function $f$ is said to be regular at $X\in\M$ along $\T_X\M$ if
	\begin{itemize}
		\item for all $V\in \T_X\M$, $f'(X;V)=\lim_{t\downarrow 0}\displaystyle \frac{f(X+tV)-f(X)}{t}$ exists, and
		\item for all $V\in \T_X\M$, $f'(X;V) = f^\circ (X;V)$.
	\end{itemize}
\end{definition}
For a smooth function $f$, we know that $\grad\,f(X)= \Proj_{\T_X\M} \nabla f(X)$  {by our choice of the Riemannian metric}. According to Theorem 5.1 in \cite{Yang-manifold-optimality-2014}, for a regular function $F$, we  have $\hat{\partial}F(X)=\Proj_{\T_X\M}(\partial F(X)).$ Moreover, the function $F(X)=f(X)+h(X)$ in problem \eqref{prob-g} is regular according to Lemma 5.1 in \cite{Yang-manifold-optimality-2014}. Therefore, we have $\hat{\partial}F(X)=\Proj_{\T_X\M}(\nabla f(X)+\partial h(X))=\grad f(X)+ \Proj_{\T_X\M}(\partial h(X))$. By Theorem 4.1 in \cite{Yang-manifold-optimality-2014}, the first-order necessary condition of problem \eqref{prob-g} is given by $0\in \grad f(X)+ \Proj_{\T_X\M}(\partial h(X))$.
\begin{definition}\label{stationary_point}
A point $X\in\M$ is called a stationary point of problem \eqref{prob-g} if it satisfies the first-order necessary condition; i.e., $0\in \grad f(X)+ \Proj_{\T_X\M}(\partial h(X))$.
\end{definition}
}
 {A classical geometric concept in the study of manifolds is that of an exponential mapping, which defines a geodesic curve on the manifold. However, the exponential mapping is difficult to compute in general. The concept of a retraction \cite{Absil2009}, which is a first-order approximation of the exponential mapping and can be more amenable to computation, is given as follows.}
\begin{definition}\label{def_retraction}\cite[Definition 4.1.1]{Absil2009}
A retraction on a differentiable manifold $\mathcal{M}$ is a smooth mapping $\Retr$ from the tangent bundle $\T\mathcal{M}$ onto $\mathcal{M}$ satisfying the following two conditions (here $\Retr_X$ denotes the restriction of $\Retr$ onto $\T_X \mathcal{M}$):
\begin{enumerate}
\item $\Retr_X(0)=X, \forall X\in\M$, where $0$ denotes the zero element of $\T_X\mathcal{M}$.
\item For any $X\in\M$, it holds that
    \[\lim_{\T_X\M\ni\xi\rightarrow 0}\frac{\|\Retr_X(\xi)-(X+\xi)\|_F}{\|\xi\|_F} = 0.\]
\end{enumerate}
\end{definition}
\begin{remark}
Since $\M$ is an embedded submanifold of $\R^{n\times r}$, we can treat $X$ and $\xi$ as elements in $\R^{n\times r}$ and hence their sum is well defined. The second condition in Definition~\ref{def_retraction} ensures that $\Retr_{X}(\xi)  =X +\xi +\mathcal{O}(\normfro{\xi}^2)$ and $D\Retr_{X}(0)=\Id$, where $D\Retr_{X}$ is the differential of $\Retr_X$ and $\Id$ denotes the identity mapping. For more details about retraction, we refer the reader to \cite{Absil2009,Boumal2016} and the references therein.
\end{remark}
The retraction onto the Euclidean space is simply the identity mapping; i.e., $\Retr_X(\xi)=X+\xi$. For the Stiefel manifold $\St(n,r)$, common retractions
include the exponential mapping \cite{EdelmanAriasSmith1999}
\[\Retr_X^{\mathrm{exp}}(t\xi) = [X,Q] \exp\left(t\begin{bmatrix}
-X^\top\xi & -R^\top \\
R&0
\end{bmatrix}\right)\begin{bmatrix}
I_r\\0
\end{bmatrix}, \]
where $QR=-(I_n-XX^\top)\xi$ is the unique QR factorization; the polar decomposition
\[\Retr_{X}^{\mathrm{polar}}(\xi)=(X+\xi)(I_r +\xi^\top\xi)^{-1/2};\]
the QR decomposition
\[\Retr_{X}^{\mathrm{QR}}(\xi)=\mathrm{qf}(X+\xi),\]
where $\mathrm{qf}(A)$ is the $Q$ factor of the QR factorization of $A$;
the Cayley transformation \cite{Wen-Yin-2013}
\[\Retr_{X}^{\mathrm{cayley}}(\xi)=\left(I_n-\frac{1}{2}W(\xi)\right)^{-1}\left(I_n+\frac{1}{2}W(\xi)\right)X,\]
where $W(\xi)=(I_n-\frac{1}{2}XX^\top)\xi X^\top-X\xi^\top(I_n-\frac{1}{2}XX^\top)$.

For any matrix $Y\in \R^{n\times r}$ with $r\leq n$, its orthogonal projection onto the Stiefel manifold $\St(n,r)$ is given by $UI_{r}V^\top$, where $U,V$ are the left and right singular vectors of $Y$, respectively. If $Y$ has full rank, then the projection can be computed by $Y(Y^\top Y)^{-1/2}$, which is the same as the polar decomposition. The total cost of computing the projection $UI_rV^\top$ is $8nr^2+\mathcal{O}(r^3)$ flops, where the SVD needs $6nr^2+\mathcal{O}(r^3)$ flops \cite{golub2012matrix} and the formation of $UI_rV^\top$ needs $2nr^2$ flops. By comparison, if $Y=X+\xi$ and $\xi\in T_X\M$, then the exponential mapping takes $8nr^2+\mathcal{O}(r^3)$ flops and the polar decomposition takes $3nr^2 +\mathcal{O}(r^3)$ flops, where $\xi^\top\xi$  needs $nr^2$ flops and the remaining $2nr^2+ \mathcal{O}(r^3)$ flops come from the final assembly. Thus, polar decomposition is cheaper than  the projection. Moreover, the QR decomposition of $X+\xi$ takes $2nr^2 + \mathcal{O}(r^3)$ flops. For the Cayley transformation of $X+\xi$, the total cost is $7nr^2+\mathcal{O}(r^3)$ \cite{Wen-Yin-2013,jiang2015framework}. In our algorithm that will be introduced later, we need to perform one retraction operation in each iteration. We need to point out that retractions may also affect the overall convergence speed of the algorithm. As a result, determining the most efficient retraction used in the algorithm is still an interesting question to investigate in practice; see also the discussion after Theorem 3 of \cite{Liu-So-Wu-2018}.



The retraction $\Retr$ has the following properties that are useful for our convergence analysis:
{{
\begin{fact}\label{retraction:property} (\cite{Boumal2016,Liu-So-Wu-2018})
Let ${\M}$ be a compact embedded submanifold of an Euclidean space. For all $X\in {\M}$ and $\xi \in \T_X\M$, there exist constants $M_1>0$ and $M_2>0$ such that the following two inequalities hold:
\be\label{first-bounded}
\normfro{R_X(\xi)-X}\leq M_1\normfro{\xi}, \quad\forall X\in {\M}, \xi \in \T_X\M,
\ee
\be\label{Second-bounded}
\normfro{R_X(\xi)-(X+\xi)}\leq M_2\normfro{\xi}^2, \quad\forall X\in {\M}, \xi \in \T_X\M.
\ee
\end{fact}
}}

\section{Proximal Gradient Method on the Stiefel Manifold}\label{sec:ManPG}

\subsection{The ManPG Algorithm}
For manifold optimization problems with a smooth objective, the Riemannian gradient method \cite{Abrudan2008,Absil2009,Nishimori2005} has been one of the main methods of choice. A generic update formula of the Riemannian gradient method for solving
\[\min_X \ F(X), \ \st, \ X\in\M\]
is
\[ 
 X_{k+1} := \Retr_{X_k}(\alpha_k V_k),
\] 
where $F$ is smooth, $V_k$ is a descent direction of $F$ in the tangent space $\T_{X_k}\M$, and $\alpha_k$ is a step size.
Recently, Boumal\etal\cite{Boumal2016} established the sublinear rate of the Riemannian gradient method for returning a point $X_k$ satisfying $\|\grad\,F(X_k)\|_F < \epsilon$. Liu\etal\cite{Liu-So-Wu-2018} proved that the Riemannian gradient method converges linearly for quadratic minimization over the Stiefel manifold. Other methods for solving manifold optimization problems with a smooth objective have also been studied in the literature, including the conjugate gradient methods \cite{Absil2009,Abrudan2009}, trust region methods \cite{Absil2009,Boumal2016}, and Newton-type methods \cite{Absil2009,Savas2010}.

 {We now develop our ManPG algorithm for solving \eqref{prob-g}. Since the objective function in \eqref{prob-g} has a composite structure, a natural idea is to extend the proximal gradient method from the Euclidean setting to the manifold setting. The proximal gradient method for solving $\min_X F(X):=f(X)+h(X)$ in the Euclidean setting generates the iterates as follows:
\be\label{alg:proximal_euclidean}
X_{k+1}:=\argmin_Y f(X_k) +\inp{\nabla f(X_k)}{Y-X_k} + \frac{1}{2t}\normfro{Y-X_k}^2 + h(Y).
\ee
In other words, one minimizes the quadratic model $Y \mapsto f(X_k) +\inp{\nabla f(X_k)}{Y-X_k} + \frac{1}{2t}\normfro{Y-X_k}^2 + h(Y)$ of $F$ at $X_k$ in the $k$-th iteration, where $t>0$ is a parameter that can be regarded as the stepsize. It is known that the quadratic model is an upper bound of $F$ when $t\leq 1/L$, where $L$ is the Lipschitz constant of $\nabla f$.
The subproblem \eqref{alg:proximal_euclidean} corresponds to the proximal mapping of $h$ and the efficiency of the proximal gradient method relies on the assumption that \eqref{alg:proximal_euclidean} is easy to solve.} 
For \eqref{prob-g}, in order to deal with the manifold constraint, we need to ensure that the descent direction lies in the tangent space. This motivates the following subproblem for finding the descent direction $V_k$ in the $k$-th iteration:
\be\label{tangent-subproblem1-R}
\ba{rl}
V_k :=\argmin_V & \inp{\grad\,f(X_k)}{V} + \frac{1}{2t}\normtwo{V}_F^2+h(X_k+V) \\
         \st   & V\in \T_{X_k}\M, 
\ea
\ee
where $t>0$ is the stepsize.  {Here and also in the later discussions, we can interpret $X_k+V$ as the sum of $X_k$ and $V$ in the ambient Euclidean space $\R^{n\times r}$, as $\M$ is an embedded submanifold of $\R^{n\times r}$.} Note that \eqref{tangent-subproblem1-R} is different from \eqref{alg:proximal_euclidean} in two places: (i) the Euclidean gradient $\nabla f$ is changed to the Riemannian gradient $\grad\,f$; (ii) the descent direction $V_k$ is restricted to the tangent space. 
Following the definition of $\grad\,f$, we have
\[ 
\inp{\grad f(X_k)}{V}=\inp{\nabla f(X_k)}{V}, \quad \forall V\in \T_{X_k}\M,
\] 
which implies that \eqref{tangent-subproblem1-R} can be rewritten as
\be\label{tangent-subproblem-R}
\ba{rl}
V_k :=\argmin_V & \inp{\nabla f(X_k)}{V} +\frac{1}{2t}\normtwo{V}_F^2+h(X_k+V)\\
\st    & V\in \T_{X_k}\M. 
\ea
\ee
As a result, we do not need to compute the Riemannian gradient $\grad\,f$. Rather, only the Euclidean gradient $\nabla f$ is needed.
Note that without considering the constraint $V\in\T_{X_k}\M$, \eqref{tangent-subproblem-R} computes a proximal gradient step. Therefore, \eqref{tangent-subproblem-R} can be viewed as a proximal gradient step restricted to the tangent space $\T_{X_k}\M$. Since for an arbitrary stepsize $\alpha_k>0$, $X_k+\alpha_kV_k$ does not necessarily lie on the manifold $\M$, we perform a retraction to bring it back to $\M$.

Our ManPG algorithm for solving \eqref{prob-g} is described in Algorithm \ref{alg:ManPG}. Note that ManPG involves an Armijo line search procedure to determine the stepsize $\alpha$. As we will show in Section \ref{sec:convergence}, this backtracking line search procedure is well defined; i.e., it will terminate after finite number of steps.

\begin{algorithm}[ht]
\caption{Manifold Proximal Gradient Method (ManPG) for Solving \eqref{prob-g}}\label{alg:ManPG}
\begin{algorithmic}[1]
\STATE{Input: initial point $X_0 \in \M$, $\gamma\in(0,1)$, stepsize $t>0$} 
\FOR{$k=0,1,\ldots$}
    \STATE{obtain $V_k$ by solving the subproblem \eqref{tangent-subproblem-R}}
    \STATE{set $\alpha=1$} 
    \WHILE{$F(\Retr_{X_k}(\alpha V_k))> F(X_k)- \displaystyle\frac{\alpha \normfro{V_k}^2}{2t}$}
        \STATE{$\alpha= \gamma\alpha$}
    \ENDWHILE
    \STATE{set $X_{k+1}=\Retr_{X_k}(\alpha V_k)$}
\ENDFOR
\end{algorithmic}
\end{algorithm}

\subsection{Regularized Semi-Smooth Newton Method for Subproblem \eqref{tangent-subproblem-R}}\label{sec:ssn}

The main computational effort of Algorithm \ref{alg:ManPG} lies in solving the convex subproblem \eqref{tangent-subproblem-R}. 
We have conducted extensive numerical experiments and found that the semi-smooth Newton method (SSN) is very suitable for this purpose. 
The notion of semi-smoothness was originally introduced by Mifflin \cite{Mifflin-1977} for real-valued functions and later extended to vector-valued mappings by Qi and Sun \cite{Qi-Sun-1993}. A pioneering work on the SSN method was due to Solodov and Svaiter \cite{Solodov1998}, in which the authors proposed a globally convergent Newton method by exploiting the structure of monotonicity and established a local superlinear convergence rate under the conditions that the generalized Jacobian is semi-smooth and non-singular at the global optimal solution. The convergence rate guarantee was later extended in \cite{Zhou2005} to the setting where the generalized Jacobian is not necessarily non-singular.  Recently, the SSN method has received significant amount of attention due to its success in solving structured convex problems to a high accuracy. In particular, it has been successfully applied to solving SDP \cite{ZhaoSunToh2008,Sun-sdpnal+}, LASSO \cite{Sun-lasso-2018}, nearest correlation matrix estimation \cite{Qi-Sun-NCM-IMA}, clustering \cite{Wang-Sun-Toh-2009}, sparse inverse covariance selection \cite{Sun-group-lasso}, and composite convex minimization \cite{Xiao2016}.

In the following we show how to apply the SSN method to solve the subproblem \eqref{tangent-subproblem-R} with $\M=\St(n,r)$. 
The tangent space to $\M=\St(n,r)$ is given by
\[\T_X \M = \{V\mid V^\top X+X^\top V=0\}.\]
For ease of notation, we define the linear operator $\A_k$ by $\A_k(V) := V^\top X_k+X_k^\top V$ and rewrite the subproblem \eqref{tangent-subproblem-R} as
\be\label{tangent-subproblem}
\ba{rl}
V_k:=\argmin_V & \inp{\nabla f(X_k)}{V} +\frac{1}{2t}\normtwo{V}_F^2+h(X_k+V)\\
\st    & \A_k(V) = 0.
\ea
\ee
By associating a Lagrange multiplier $\Lambda$ to the linear equality constraint, the Lagrangian function of \eqref{tangent-subproblem} can be written as
\[ 
\LCal(V;\Lambda) = \inp{\nabla f(X_k)}{V} +\frac{1}{2t}\normtwo{V}_F^2+h(X_k+V) - \inp{\A_k(V)}{\Lambda},
\] 
and the KKT system of \eqref{tangent-subproblem} is given by
\be\label{tangent-subproblem-kkt}
0 \in \partial_V \LCal(V;\Lambda), \quad \A_k(V) = 0.
\ee
The first condition in \eqref{tangent-subproblem-kkt} implies that $V$ can be computed by
\be\label{compute-D}
V(\Lambda) = \prox_{th}(B(\Lambda))-X_k \quad\mbox{with}\quad B(\Lambda) = X_k-t(\nabla f(X_k) - \A_k^*(\Lambda)),
\ee
where $\A_k^*$ denotes the adjoint operator of $\A_k$.
By substituting \eqref{compute-D} into the second condition in \eqref{tangent-subproblem-kkt}, we see that $\Lambda$ satisfies
\be\label{sub_VI}
E(\Lambda)\equiv \A_k(V(\Lambda))=V(\Lambda)^\top X_k+X_k^\top V(\Lambda)=0.
\ee
We will use the SSN method to solve \eqref{sub_VI}.  {To do so, we need to first show that the operator $E$ is monotone and Lipschitz continuous}. For any $\Lambda_1, \Lambda_2\in S^r$, we have
\be\label{Lipschitz_E}\bad
& \normfro{ E(\Lambda_1)-E(\Lambda_2)}\\
\leq & \normtwo{\A_k}_{op} \normfro{\prox_{ th}(B(\Lambda_1))-\prox_{ th}(B(\Lambda_2))}\\
\leq& \normtwo{\A_k}_{op} \normfro{B(\Lambda_1)-B(\Lambda_2)}\\
\leq& t \normtwo{\A_k}_{op}^2 \normfro{\Lambda_1-\Lambda_2},
\ead\ee
where the second inequality holds since the proximal mapping is  non-expansive.  Moreover,
\begin{align*}
  & \inp{E(\Lambda_1)-E(\Lambda_2)}{\Lambda_1-\Lambda_2}\\
= & \inp{V(\Lambda_1)-V(\Lambda_2)}{\A_k^*(\Lambda_1-\Lambda_2)}\\
= & \frac{1}{t}\inp{\prox_{ th}(B(\Lambda_1))-\prox_{ th}(B(\Lambda_2))}{B(\Lambda_1)-B(\Lambda_2)}\\
\geq & \frac{1}{t}\normfro{\prox_{ th}(B(\Lambda_1))-\prox_{ th}(B(\Lambda_2))}^2\\
\geq &  \frac{1}{t\normtwo{\A_k}_{op}^2} \normfro{E(\Lambda_1)-E(\Lambda_2)}^2  \geq 0,
\end{align*}
where the first inequality holds since the proximal mapping is firmly non-expansive and the second inequality is due to \eqref{Lipschitz_E}. In particular, we see that $E$ is actually $1/(t\normtwo{\A_k}_{op}^2)$-coercive. Therefore, $E$ is indeed monotone and Lipschitz continues, and we can apply the SSN method to find a zero of $E$. In order to apply the SSN method, we need to compute the generalized Jacobian of $E$.\footnote{See Appendix~\ref{sec:semi-smooth} for a brief discussion of the semi-smoothness of operators related to the proximal mapping.} Towards that end, observe that the vectorization of $E(\Lambda)$ can be represented by
\begin{align*}
\vet(E(\Lambda))&=(X_k^\top\otimes I_p) K_{nr}\vet(V(\Lambda))+(I_r\otimes X_k^\top) \vet(V(\Lambda))\\
&=(K_{rr}+I_{r^2})(I_p\otimes X_k^\top)\left[\prox_{th(\cdot)}(\vet(X_k-t\nabla f(X_k))+2t(I_r\otimes X_k)\vet(\Lambda))-\vet(X_k)\right],
\end{align*}
where $K_{nr}$ and $K_{rr}$ denote the commutation matrices. 
%
We define the matrix
\[\G(\vet(\Lambda))=2t (K_{rr}+I_{r^2})(I_r\otimes X_k^\top)\J(y)|_{y=\vet(B(\Lambda))} (I_r\otimes X_k),\]
where $\J(y)$ is the generalized Jacobian of $\prox_{th}(y)$. 
From \cite[Example 2.5]{hiriart1984generalized}, we know that
$\G(\vet(\Lambda))\xi = \partial \vet(E(\vet(\Lambda))\xi, \ \forall \xi\in \R^{r^2}$.
Thus, $\G(\vet(\Lambda))$ can serve as a representation of
$\partial\vet(E(\vet(\Lambda)))$. 
Note that since $\Lambda$ is a symmetric matrix, we only need to focus on the lower triangular part of $\Lambda$. It is known that there exists a unique $r^2\times \frac{1}{2}r(r+1)$ matrix $U_r$, called the duplication matrix \cite[Ch 3.8]{Magnus1988}, such that $U_r \svec(\Lambda)=\vvec(\Lambda)$. The Moore-Penrose inverse of $U_r$ is $U_r^+=(U_r^\top U_r)^{-1}U_r^\top$ and satisfies $U_r^+ \vet(\Lambda)=\svec(\Lambda)$. Note that both $U_r$ and $U_r^+$ have only $r^2$ nonzero elements. As a result, we can represent the generalized Jacobian of $\svec(E(U_r\svec(\Lambda)))$ by
\[ 
\G(\svec(\Lambda))=tU_r^+ \G(\vet(\Lambda))U_r=4tU_r^+ (I_r\otimes X_k^\top)\J(y)|_{y=\vet(B(\Lambda))} (I_r\otimes X_k)U_r,
\] 
where we use the identity $K_{rr}+I_{r^2}=2U_rU_r^+$.
It should be pointed out that $\G(\svec(\Lambda))$ can be singular. Therefore, the vanilla SSN method cannot be applied directly and we need to resort to a regularized SSN method proposed in \cite{Solodov1998} and further studied in \cite{Zhou2005,Xiao2016}. It is known that the global convergence of the regularized SSN method is guaranteed if any element in $\G(\svec(\Lambda))$ is positive semidefinite \cite{Xiao2016}, which is the case here because it can be shown that $\G(\svec(\Lambda))+\G(\svec(\Lambda))^\top$ is positive semidefinite. We find that the adaptive regularized SSN (ASSN) method proposed in \cite{Xiao2016} is very suitable for solving \eqref{sub_VI}. The ASSN method first computes the Newton direction $d_k$ by solving
\be\label{newton-direction}
(\G(\svec(\Lambda_k)) + \eta I)d = -\svec(E(\Lambda_k)),
\ee
where $\eta>0$ is a regularization parameter. {If the matrix size is large, then \eqref{newton-direction} can be solved inexactly by the conjugate gradient method.}
The authors then designed a strategy to decide whether to accept this $d_k$ or not. Roughly speaking, if there is a sufficient decrease from $\|E(\Lambda_{k})\|_2$ to $\|E(\Lambda_{k+1})\|_2$, then we accept $d^k$ and set
\[\svec(\Lambda_{k+1}) = \svec(\Lambda_k) + d_k.\]
Otherwise, a safeguard step is taken. For more details on the ASSN method, we refer the reader to \cite{Xiao2016}.

\section{Global Convergence and Iteration Complexity}\label{sec:convergence}

In this section, we analyze the convergence and iteration complexity of our ManPG algorithm (Algorithm \ref{alg:ManPG}) for solving \eqref{prob-g}. 
 {Our convergence analysis consists of three steps. First, in Lemma \ref{tangent-des} we show that $V_k$ in \eqref{tangent-subproblem-R} is a descent direction for the objective function in \eqref{tangent-subproblem-R}.
Second, in Lemma \ref{sufficient_des} we show that $V_k$ is also a descent direction for the objective function in \eqref{prob-g} after applying a retraction to it; i.e., there is a sufficient decrease from $F(X_k)$ to $F(\Retr_{X_k}({\alpha}V_k))$. This is motivated by a similar result in Boumal\etal\cite{Boumal2016}, which states that the pullback function $\hat{F}(V):=F(\Retr_{X}(V))$ satisfies certain Lipschitz-type property. Therefore, the results here can be seen as an extension of the ones for smooth problems in \cite{Boumal2016} to the nonsmooth problem \eqref{prob-g}. Third, we establish the global convergence of ManPG in Theorem \ref{thm:complexity}. 
}


Now, let us begin our analysis. The first observation is that the objective function in \eqref{tangent-subproblem-R} is strongly convex, which implies that the subproblem \eqref{tangent-subproblem-R} is also strongly convex (recall that a function $g$ is said to be $\alpha$-strongly convex on $\R^{n\times r}$ if
\[ g(Y)\geq g(X) +\inp{\partial g(X)}{Y-X}+\frac{\alpha}{2}\normfro{Y-X}^2, \quad \forall X,Y\in\R^{n\times r}.) \]
The following lemma shows that $V_k$ obtained by solving \eqref{tangent-subproblem-R} is indeed a descent direction in the tangent space to $\M$ at $X_k$:
\begin{lemma}\label{tangent-des}
Given the iterate $X_k$, let
\be \label{eq:g-def}
g(V):=\inp{\nabla f(X_k)}{ V}+\frac{1}{2t}\normtwo{V}_F^2+h(X_k+ V)
\ee
denote the objective function in \eqref{tangent-subproblem-R}. Then, the following holds for any $\alpha \in[0,1]$:
\be\label{d1} g(\alpha V_k)-g(0)\leq \frac{(\alpha-2)\alpha}{2t} \normfro{V_k}^2.\ee
\end{lemma}
\begin{proof}
Since $g$ is $(1/t)$-strongly convex, we have
\be\label{e1}
g(\hat{V})\geq g(V) + \inp{\partial g(V)}{\hat{V}-V} + \frac{1}{2t}\normfro{\hat{V}-V}^2, \quad \forall V, \hat{V} \in \R^{n\times r}.
\ee
In particular, if $V,\hat{V}$ are feasible for \eqref{tangent-subproblem-R} (i.e., $V,\hat{V}\in \T_{X_k}\M$), then
\[ \inp{\partial g(V)}{\hat{V}-V} = \inp{\Proj_{\T_{X_k}\M}\partial g(V)}{\hat{V}-V}. \]
From the optimality condition of \eqref{tangent-subproblem-R}, we have $0\in \Proj_{\T_{X_k}\M}\partial g(V_k)$. Letting $V = V_k$ and $\hat{V} = 0$ in \eqref{e1} yields
\[g(0)\geq g(V_k) +  \frac{1}{2t}\normfro{V_k}^2, \]
which implies that
\[ 
h(X_k)\geq\inp{\nabla f(X_k)}{V_k}+\frac{1}{2t}\normfro{V_k}^2+h(X_k+V_k) + \frac{1}{2t}\normfro{V_k}^2.
\] 
Moreover, the convexity of $h$ yields
\[ 
h(X_k+\alpha V_k)= h(\alpha(X_k+V_k)+(1-\alpha)X_k) \leq \alpha h(X_k+V_k) + (1-\alpha)h(X_k).
\] 
Upon combining the above inequalities, we obtain
\[
g(\alpha V_k) -g(0) = \inp{\nabla f(X_k)}{\alpha V_k}+\frac{\normfro{\alpha V_k}^2}{2t} + h(X_k+\alpha V_k) - h(X_k) \leq  \frac{\alpha^2-2\alpha}{2t} \normfro{V_k}^2,
\]
as desired.
\end{proof}

The following lemma shows that $\{F(X_k)\}$ is monotonically decreasing, where $\{X_k\}$ is generated by Algorithm \ref{alg:ManPG}.
\begin{lemma}\label{sufficient_des}
For any $t>0$, there exists a constant $\bar{\alpha}>0$ such that for any $0<\alpha \leq \min\{1,\bar{\alpha}\}$, the line search procedure in Algorithm \ref{alg:ManPG} is well defined, and the sequence $\{X_k\}$ generated by Algorithm \ref{alg:ManPG} satisfies
\[ 
F(X_{k+1}) -F(X_k) \leq -\frac{\alpha}{2t}\normtwo{V_k}_F^2.
\] 
\end{lemma}

\begin{proof}
Let $X_k^+= X_k+\alpha V_k$. By~\eqref{first-bounded},~\eqref{Second-bounded} and the $L$-Lipschitz continuity of $\nabla f$, for any $\alpha>0$, we have
\be\label{ineq1-0}
\bad
f(\Retr_{X_k}(\alpha V_k)) - f(X_k) &\leq \inp{\nabla f(X_k)}{\Retr_{X_k}(\alpha V_k)-X_k}+\frac{L}{2}\normfro{\Retr_{X_k}(\alpha V_k)-X_k}^2\\
&= \inp{\nabla f(X_k)}{\Retr_{X_k}(\alpha V_k)-X_k^+ +X_k^+ -X_k}+\frac{L}{2}\normfro{\Retr_{X_k}(\alpha V_k)-X_k}^2\\
&\leq M_2\normfro{\nabla f(X_k)}\normfro{\alpha V_k}^2 + \alpha\inp{\nabla f(X_k)}{V_k} + \frac{LM_1^2}{2}\normfro{\alpha V_k}^2.
\ead
\ee
Since $\nabla f$ is continuous on the compact manifold $\M$, there exists a constant $G>0$ such that $\normfro{\nabla f(X)}\leq G$ for all $X\in\M$. It then follows from \eqref{ineq1-0} that
\be\label{smooth_ine}
f(\Retr_{X_k}(\alpha V_k)) - f(X_k) \leq  \alpha\inp{\nabla f(X_k)}{V_k}+c_0\alpha^2\normfro{V_k}^2 ,
\ee
where $c_0=M_2G +LM_1^2/2$. This implies that
\[
\bad
F(\Retr_{X_k}(\alpha V_k)) - F(X_k) 
                 & \overset{\eqref{smooth_ine}}{\leq} \alpha\inp{\nabla f(X_k)}{V_k} +c_0 \alpha^2 \normfro{V_k}^2+ h(\Retr_{X_k}(\alpha V_k))-h(X_k^+)+ h(X_k^+)- h(X_k)\\
                 &\,\, \leq \,\,\,\alpha\inp{\nabla f(X_k)}{V_k} +c_0 \alpha^2 \normtwo{V_k}_F^2+ L_h\normfro{\Retr_{X_k}(\alpha V_k)-X_k^+}+  h(X_k^+)- h(X_k)\\
                 & \overset{\eqref{Second-bounded}}{\leq}(c_0  + L_hM_2) \normtwo{\alpha V_k}_F^2 + g(\alpha V_k) - \frac{1}{2t}\normfro{\alpha V_k}^2 - g(0)\\
                 & \overset{\eqref{d1}}{\leq} \left( c_0+L_hM_2 - \frac{1}{\alpha t} \right) \normfro{\alpha V_k}^2,
\ead
\]
where $g$ is defined in~\eqref{eq:g-def} and the second inequality follows from the Lipschitz continuity of $h$. Upon setting $\bar{\alpha} = 1/{(2(c_0+L_h M_2)t)}$, we conclude that for any $0< \alpha \leq \min\{\bar{\alpha},1\}$,
\[F(\Retr_{X_k}(\alpha V_k)) -F(X_k) \leq -\frac{1}{2\alpha t}\normfro{\alpha V_k}^2 = -\frac{\alpha}{2t}\normfro{V_k}^2.\]
This completes the proof.
\end{proof}

The following lemma shows that if one cannot make any progress by solving \eqref{tangent-subproblem-R} (i.e., $V_k=0$), then a stationary point is found.
\begin{lemma}\label{first_order_opt}
If $V_k=0$, then $X_k$ is a stationary point of problem \eqref{prob-g}.
\end{lemma}

\begin{proof}
By Theorem 4.1 in \cite{Yang-manifold-optimality-2014}, the optimality conditions of the subproblem \eqref{tangent-subproblem1-R} are given by
\[ 0\in  \frac{1}{t} V_k + \grad\,f(X_k)+\Proj_{\T_{X_k}\M} \partial h(X_k+V_k), \quad \ V_k\in \T_{X_k} \M. \]
If $V_k=0$, then $0\in\grad\,f(X_k)+\Proj_{\T_{X_k}\M} \partial h(X_k)$, which is exactly the first-order necessary condition of problem \eqref{prob-g} since $X_k\in \M$.
\end{proof}

From Lemma \ref{first_order_opt}, we know that $V_k=0$ implies the stationarity of $X_k$ with respect to \eqref{prob-g}. This motivates the following definition of an $\epsilon$-stationary point of \eqref{prob-g}:

\begin{definition}\label{def-epsilon-stationary}
We say that $X_k \in \M$ is an $\epsilon$-stationary point of \eqref{prob-g} if the solution $V_k$ to \eqref{tangent-subproblem} with $t=1/L$ satisfies $\normfro{V_k}\leq \epsilon/L$.
\end{definition}

We use $\normfro{V_k}\leq \epsilon/L$ as the stopping criterion of Algorithm \ref{alg:ManPG} with $t=1/L$. From Lemma \ref{sufficient_des}, we obtain the following result which is similar to the one in \cite[Theorem 2]{Boumal2016} for manifold optimization with smooth objectives.
\begin{theorem}\label{thm:complexity}
Every limit point of the sequence $\{X_k\}$ generated by Algorithm \ref{alg:ManPG} is a stationary  point of problem \eqref{prob-g}. Moreover, Algorithm \ref{alg:ManPG} with $t = 1/L$ will return an $\epsilon$-stationary point of~\eqref{prob-g} in at most $\left\lceil 2L(F(X_0)-F^*)/(\gamma\bar{\alpha}\epsilon^2) \right\rceil$ iterations, where $\bar{\alpha}$ is defined in Lemma \ref{sufficient_des} and $F^*$ is the optimal value of~\eqref{prob-g}.
\end{theorem}

\begin{proof}
Since $F$ is bounded below on $\M$, by Lemma \ref{sufficient_des}, we have
$$\lim_{k\rightarrow \infty} \normtwo{V_k}_F^2 =0.$$
Combining with Lemma \ref{first_order_opt}, it follows that every limit point of $\{X_k\}$ is a stationary point of \eqref{prob-g}. Moreover, since
$\M$ is compact, the sequence $\{X_k\}$ has at least one limit point. Furthermore, suppose that Algorithm \ref{alg:ManPG} with $t=1/L$ does not terminate after $K$ iterations; i.e., $\normfro{V_k} > \epsilon/L$ for all $k=0,1,\ldots,K-1$. Let $\alpha_k$ be the stepsize in the $k$-th iteration; i.e., $X_{k+1}=\Retr_{X_k}(\alpha_k V_k)$. From Lemma \ref{sufficient_des}, we know that $\alpha_k\geq\gamma \bar{\alpha}$. Thus, we have
\[ F(X_0)-F^*\geq F(X_0)-F(X_K) \geq \frac{t}{2}\sum_{k=0}^{K-1}\alpha_k\normfro{V_k/t}^2 > \frac{t\epsilon^2}{2} \sum_{k=0}^{K-1}\alpha_k \geq \frac{tK\epsilon^2}{2} \gamma \bar{\alpha}. \]
Therefore, Algorithm \ref{alg:ManPG} finds an $\epsilon$-stationary point in at most $\left\lceil 2L(F(X_0)-F^*)/(\gamma\bar{\alpha}\epsilon^2)\right\rceil$ iterations.
\end{proof}

{
\begin{remark}
When the objective function $F$ in \eqref{prob-g} is smooth (i.e., the nonsmooth function $h$ vanishes), the iteration complexity in Theorem \ref{thm:complexity} matches the result given by Boumal\etal in \cite{Boumal2016}. Zhang and Sra \cite{Sra-1st-order-geodesically-convex-2016} analyzed the iteration complexity of some first-order methods, but they assumed that the objectives are geodesically convex. Such an assumption is rather restrictive, as it is known that every smooth function that is geodesically convex on a compact Riemannian manifold is constant \cite{bishop1969manifolds}. Bento\etal\cite{Bento-iteration-complexity-2017} also established some iteration complexity results for gradient, subgradient, and proximal point methods. However, their results for gradient and subgradient methods require the objective function to be convex and the manifold to be of nonnegative curvature, while those for proximal point methods only apply to convex objective functions over the Hadamard manifold.
\end{remark}
}

\section{Numerical Experiments}\label{sec:numerical}

In this section, we apply our ManPG algorithm (Algorithm \ref{alg:ManPG}) to solve the sparse PCA \eqref{spca} and compressed modes (CM) \eqref{CM} problems. We compare ManPG with two existing methods SOC \cite{Lai-Osher-soc-2014} and PAMAL \cite{Chen2016}.
For both problems, we set the parameter $\gamma=0.5$ and use the polar decomposition as the retraction mapping in ManPG.  {The latter is because it is found that the MATLAB implementation of QR factorization is slower than the polar decomposition; see \cite{absil2015low}.}
Moreover, we implement a more practical version of ManPG, named ManPG-Ada and described in Algorithm \ref{alg:ManPG-adap}, that incorporates a few tricks including adaptively updating the stepsize $t$. We set the parameters $\gamma = 0.5$ and $\tau = 1.01$ in ManPG-Ada. All the codes used in this section were written in MATLAB and run on a standard PC with 3.70 GHz I7 Intel microprocessor and 16GB of memory.

\subsection{A More Practical ManPG: ManPG-Ada}
 {In this subsection, we introduce some tricks used to further improve the performance of ManPG in practice. First, a warm-start strategy is adopted for SSN; i.e., the initial point $\Lambda_0$ in SSN is set as the solution of the previous subproblem. {{For the ASSN algorithm, we always take the semi-smooth Newton step as suggested by \cite{Xiao2016}.}} 
Second, 
we adaptively update $t$ in ManPG. When $t$ is large, we may need smaller total number of iterations to reach an $\epsilon$-stationary point. However, it increases the number of line search steps and the SSN steps. For sparse PCA and CM problems, we found that setting $t=1/L$ leads to fewer number of line search steps. We can then increase $t$ slightly if no line search step was needed in the previous iteration. This new version of ManPG, named ManPG-Ada, is described in Algorithm \ref{alg:ManPG-adap}. 
We also applied ManPG-Ada to solve sparse PCA and CM problems and compared its performance with ManPG, SOC, and PAMAL.
	 \begin{algorithm}[ht]
		\caption{ManPG-Ada for Solving \eqref{prob-g}}\label{alg:ManPG-adap}
		\begin{algorithmic}[1]
			\STATE{Input: initial point $X_0 \in \M$, $\gamma\in(0,1)$, $\tau>1$ and Lipschitz constant $L$}
			\STATE{set $t=1/L$}
			\FOR{$k=0,1,\ldots$}
			\STATE{obtain $V_k$ by solving the subproblem \eqref{tangent-subproblem-R}}
			\STATE{set $\alpha=1$ and $\text{linesearchflag} = 0$} 
			\WHILE{$F(\Retr_{X_k}(\alpha V_k))> F(X_k)-\displaystyle\frac{\alpha  \normfro{V_k}^2}{2t}$}
			\STATE{$\alpha= \gamma\alpha$}
			\STATE{$\text{linesearchflag} = 1$}
			\ENDWHILE
			\STATE{set $X_{k+1}=\Retr_{X_k}(\alpha V_k)$}
			\IF{$\text{linesearchflag} = 1$ }
			\STATE{ $t = \tau t$}
			\ELSE
			\STATE{ $t = \max\{1/L,t/\tau\}$}
			\ENDIF
			\ENDFOR
		\end{algorithmic}
\end{algorithm}}

\subsection{Numerical Results on CM}\label{sec:numerical:CMS}

For the CM problem \eqref{CM}, both SOC~\cite{Lai-Osher-soc-2014} and PAMAL~\cite{Chen2016} rewrite the problem as
\be\label{CM-re}
\ba{ll}
\min_{X,Q,P\in\br^{n\times r}} & \Tr(P^\top H P) + \mu\|Q\|_1 \\
\st  & Q = P, X = P, X^\top X = I_r.
\ea
\ee
SOC employs a three-block ADMM to solve \eqref{CM-re}, which updates the iterates as follows:
\be\label{soc-CM-re}
\ba{ll}
P_{k+1} & := \argmin_P \ \Tr(P^\top H P) + \frac{\beta}{2}\|P-Q_k + \Lambda_k\|_F^2 + \frac{\beta}{2}\|P-X_k+\Gamma_k\|_F^2, \\
Q_{k+1} & := \argmin_Q \ \mu\|Q\|_1 + \frac{\beta}{2}\|P_{k+1}-Q + \Lambda_k\|_F^2, \\
X_{k+1} & := \argmin_X \ \frac{\beta}{2}\|P_{k+1}-X+\Gamma_k\|_F^2, \ \st, X^\top X = I_r, \\
\Lambda_{k+1} & := \Lambda_k + P_{k+1}-Q_{k+1}, \\
\Gamma_{k+1} & := \Gamma_k + P_{k+1}-X_{k+1}.
\ea
\ee

PAMAL uses an inexact augmented Lagrangian method to solve \eqref{CM-re} with the augmented Lagrangian function being minimized by the proximal alternating minimization algorithm proposed in \cite{Attouch-2016-MP}.  {Both SOC and PAMAL need to solve a linear system $(H + \beta I)X= B$, where $B$ is a given matrix.}

In our numerical experiments, we tested the same problems as in \cite{Ozolins2013} and \cite{Chen2016}. In particular, we consider the time-independent Schr\"odinger equation
\[\hat{H}\phi(x)=\lambda\phi(x),x\in\Omega,\]
where $\hat{H} = -\half\Delta$ denotes the Hamiltonian, $\Delta$ denotes the Laplacian operator, and $H$ is a symmetric matrix formed by discretizing the Hamiltonian $\hat{H}$. We focus on the 1D free-electron (FE) model. The FE model describes the behavior of valence electron in a crystal structure of a metallic solid and has $\hat{H}=-\half\partial_x^2$. We consider the system on a domain $\Omega = [0,50]$ with periodic boundary condition and discretize the domain with $n$ equally spaced nodes. The stepsize $t$ in Algorithm \ref{alg:ManPG} was set to $1/(2\lambda_{\max}(\hat{H}))$, where $\lambda_{\max}(\hat{H})$ denotes the largest eigenvalue of $\hat{H}$ and is given by $2n^2/50^2$ in this case.

Since the matrix $H$ is circulant, we used FFT to solve the linear systems in SOC and PAMAL, which is more efficient than directly inverting the matrices. We terminated ManPG when $\normfro{V_k/t}^2\leq \epsilon:= 10^{-8}nr$ or the maximum iteration number $30000$ was reached. For the inner iteration of ManPG (i.e., using SSN to solve \eqref{tangent-subproblem-R}), we terminated it when $\normfro{E(\Lambda)}^2\leq \max\{10^{-13}, \min\{10^{-11},10^{-3} t^2 \epsilon\}\}$ or the maximum iteration number $100$ was reached. In all the tests of the CM problem, we ran ManPG first and denote $F_M$ as the returned objective value. We then ran SOC and PAMAL and terminated them when $F(X_k)\leq F_M + 10^{-7}$ and
\be\label{stop-soc-pamal}
\frac{\|Q_k-P_k\|_F}{\max\{1,\|Q_k\|_F,\|P_k\|_F\}} + \frac{\|X_k-P_k\|_F}{\max\{1,\|X_k\|_F,\|P_k\|_F\}} \leq 10^{-4}.
\ee
Note that \eqref{stop-soc-pamal} measures the constraint violation of the reformulation \eqref{CM-re}. If \eqref{stop-soc-pamal} was not satisfied in $30000$ iterations, then we terminated SOC and PAMAL.  {We also ran ManPG-Ada (Algorithm \ref{alg:ManPG-adap}) and terminated it if $F(X_k)\leq F_M + 10^{-7}$. In our experiments, we found that SOC and PAMAL are very sensitive to the choice of parameters. The default setting of the parameters of SOC and PAMAL suggested in \cite{Ozolins2013} and \cite{Chen2016} usually cannot achieve our desired accuracy. Unfortunately, there is no systematic study on how to tune these parameters. We spent a significant amount of effort on tuning these parameters, and the ones we used are given as follows. For SOC \eqref{soc-CM-re}, we set the penalty parameter $\beta = nr\mu/25 + 1$. For PAMAL, we found that the setting of the parameters given on page B587 of \cite{Chen2016} does not work well for the problems we tested. Instead, we found that the following settings of these parameters worked best and they are thus adopted in our tests: $\tau=0.99$, $\gamma=1.001$, $\rho^{1}=2\left|\lambda_{\min }(H)\right|+r / 10 + 2,  \overline{\Lambda}_{p, \min }=-100, \overline{\Lambda}_{p, \max }=100, \Lambda_{p}^{1}=0_{nr},p=1,2,$ and $\epsilon^{k}=(0.995)^{k}$, $k \in \mathbb{N}$. For the meanings of these parameters, we refer the reader to page B587 of \cite{Chen2016}.
We used the same parameters of PAM in PAMAL as recommended by \cite{Chen2016}. For different settings of $(n,r,\mu)$, we ran the four algorithms with 50 instances whose initial points were obtained by projecting randomly generated points onto $\St(n,r)$. Since problem \eqref{CM} is nonconvex, it is possible that ManPG, ManPG-Ada, SOC and PAMAL return different solutions from random initializations. To increase the chance that all four solvers found the same solution, we ran the Riemannian subgradient method for 500 iterations and used the resulting iterate as the initial point. The Riemannian subgradient method is described as follows:
\be\label{alg:rieman_subgrad}
\bad
\hat{\partial}F(X_k) & := \Proj_{\T_{X_k}\St(n,r)}(2HX_k + \mu\text{sign}(X_k)), \\
X_{k+1} & := \Retr_{X_k} \left( -\frac{1}{k^{3/4}} \hat{\partial}F(X_k) \right),
\ead
\ee
where $\text{sign}(\cdot)$ denotes the element-wise sign function. Moreover, we tried to run the Riemannian subgradient method \eqref{alg:rieman_subgrad} until it solved the CM problem. However, this method is extremely slow and we only report one case in Figure \ref{figure:CMS_dim_sub}. We report the averaged CPU time, iteration number, and sparsity in Figures \ref{figure:CMS_dim_sub} to \ref{figure:CMS_sp_mu}, where sparsity is the percentage of zeros and when computing sparsity, $X$ is truncated by zeroing out its entries whose magnitude is smaller than $10^{-5}$. For SOC and PAMAL, we only took into account the solutions that were close to the one generated by ManPG. Here the closeness of the solutions is measured by the distance between their column spaces. More specifically, let $X_M$, $X_S$, and $X_P$ denote the solutions generated by ManPG, SOC, and PAMAL, respectively. Then, their distances are computed by $\text{dist}(X_M,X_S)=\normfro{X_M X_M^\top-X_S X_S^\top}$ and $\text{dist}(X_M,X_P)=\normfro{X_M X_M^\top-X_P X_P^\top}$. We only counted the results if $\text{dist}^2(X_M,X_S)\leq 0.1$ and $\text{dist}^2(X_M,X_P)\leq 0.1$. In Figure \ref{figure:CMS_dim_sub}, we report the results of Riemannian subgradient method with respect to different $n$'s. We terminated the Riemannian subgradient method \eqref{alg:rieman_subgrad} if $F(X_k)<F_M + 10^{-3}$. We see that this accuracy tolerance $10^{-3}$ is too large to yield a good solution with reasonable sparsity level, yet it is already very time consuming. As a result, we do not report more results on the Riemannian subgradient method. In Figures \ref{figure:CMS_dim}, \ref{figure:CMS_rank}, and \ref{figure:CMS_sp_mu}, we see that the solutions returned by ManPG and ManPG-Ada have better sparsity than SOC and PAMAL. We also see that ManPG-Ada outperforms ManPG in terms of CPU time and iteration number. In Figure \ref{figure:CMS_dim}, the iteration number of ManPG increases with the dimension $n$, because the Lipschitz constant $L = 2\lambda_{\max}(H) = {4n^2}/{50^2}$ increases quadratically, which is consistent with our complexity result. In Figure \ref{figure:CMS_rank}, we see that the CPU times of ManPG and ManPG-Ada are comparable to SOC and PAMAL when $r$ is small, but are slightly more when $r$ gets large. In Figure \ref{figure:CMS_sp_mu}, we see that the performance of the algorithms is also affected by $\mu$. In terms of CPU time, ManPG and ManPG-Ada are comparable to SOC and PAMAL when $\mu$ gets large. 
We also report the total number of line search steps and the averaged iteration number of SSN in ManPG and ManPG-Ada in Table \ref{tab:CMs_ave_line_SSN}. We see that ManPG-Ada needs more line search steps and SSN iterations, but as we show in Figures \ref{figure:CMS_dim}, \ref{figure:CMS_rank}, and \ref{figure:CMS_sp_mu}, ManPG-Ada is faster than ManPG in terms of CPU time. This is mainly because the computational costs of retraction and SSN steps in this problem are both nearly the same as computing the gradient. In the last two columns of Table \ref{tab:CMs_ave_line_SSN}, `\#s$\mid$d' denotes the number of instances for which SOC and PAMAL generate same, different solutions as ManPG with the closeness measurement discussed above; `\# f' denotes the number of instances that SOC and PAMAL fail to converge. We see that for the tested instances of the CM problem, all algorithms converged thanks to the parameters that we chose, although sometimes the solutions generated by PAMAL are different from those generated by ManPG and SOC.
}
\begin{figure}[H]
	\begin{center}
		\minipage{0.33\textwidth}
		\subfigure[CPU ]{
			\centerline{\includegraphics[width=0.9\linewidth]{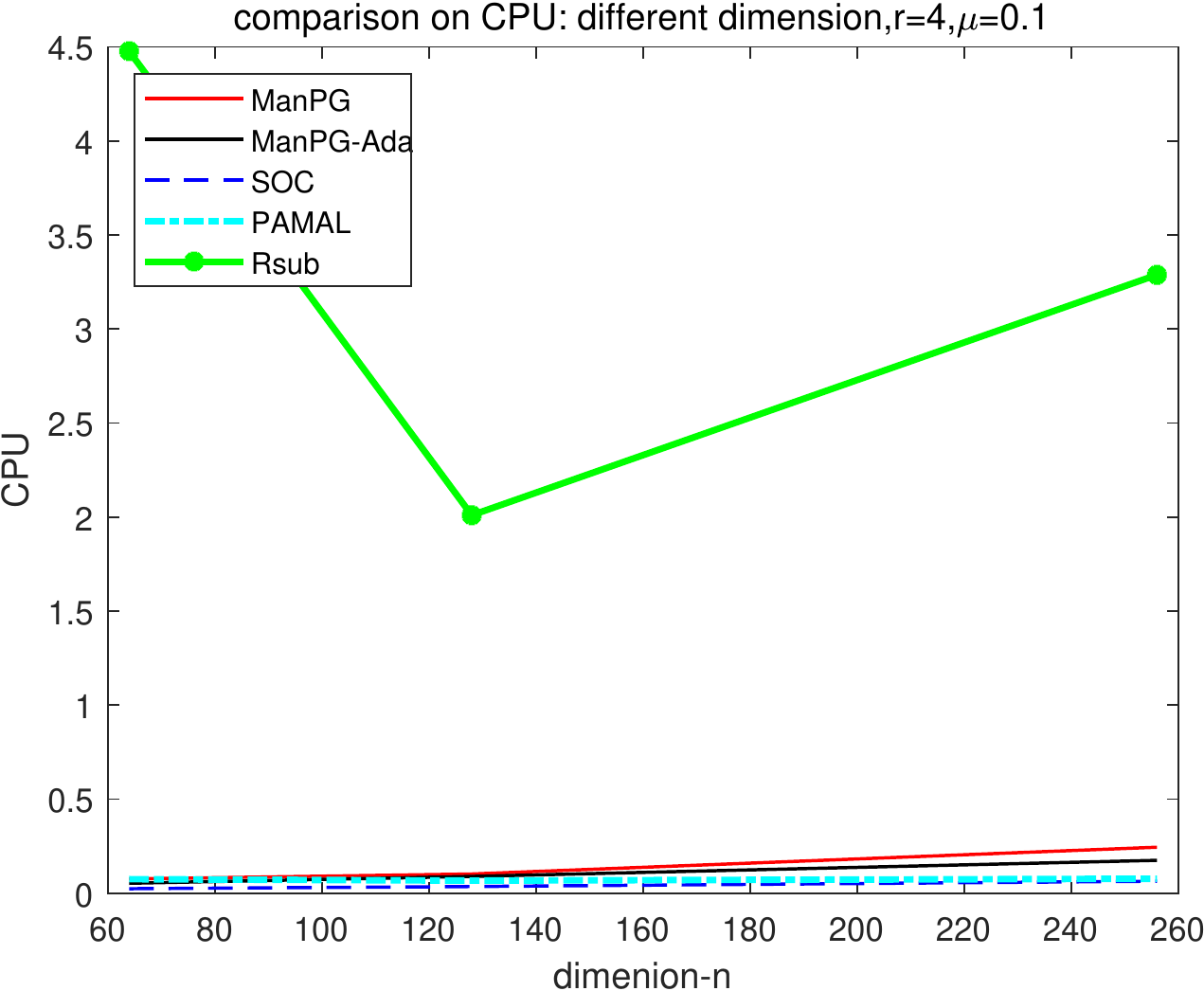}}}
		\endminipage\hfill
		\minipage{0.33\textwidth}
		\subfigure[ Iteration ]{\includegraphics[width=0.9\linewidth]{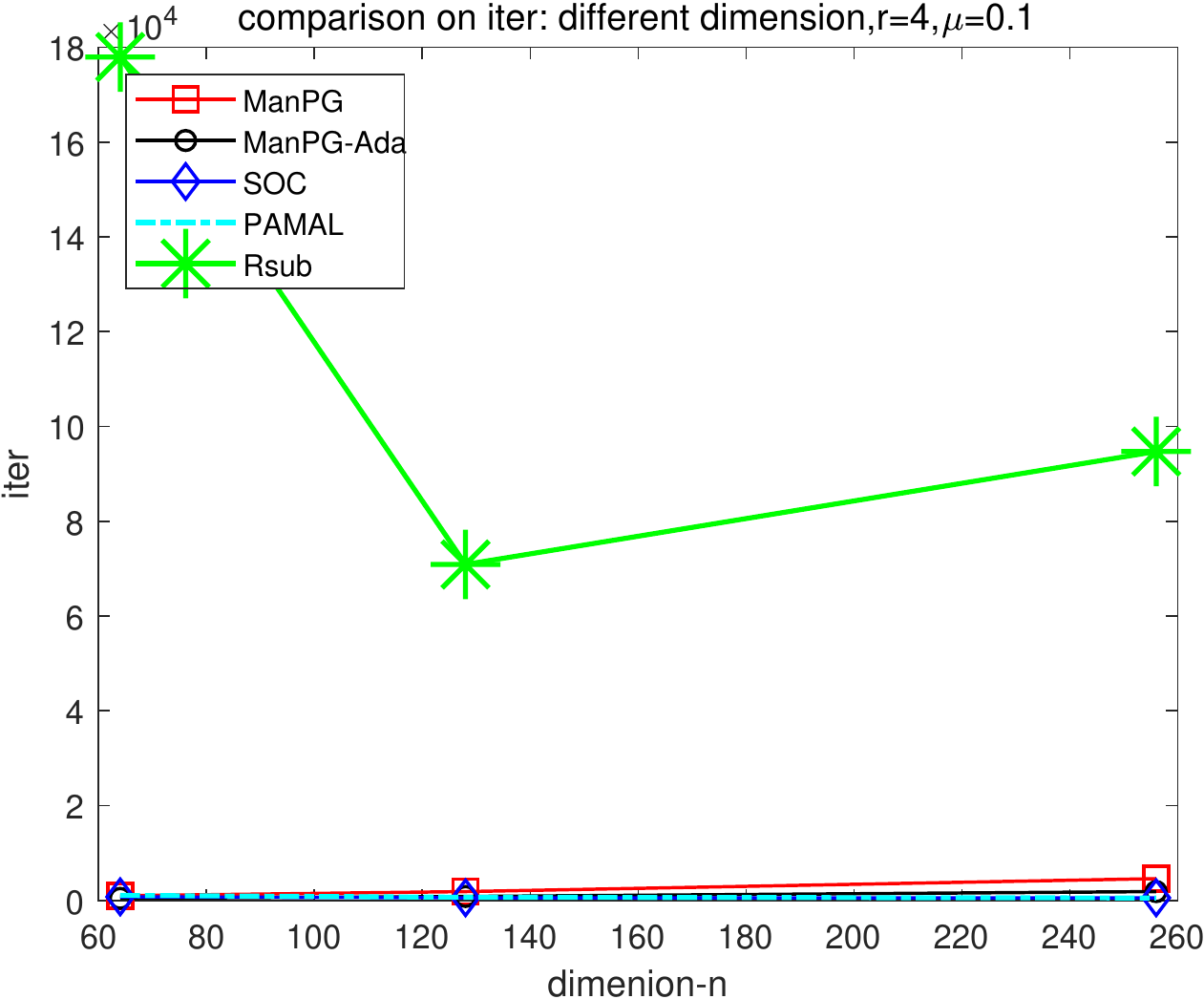}}
		\endminipage\hfill
		\minipage{0.33\textwidth}
		\subfigure[Sparsity ]{\includegraphics[width=0.9\linewidth]{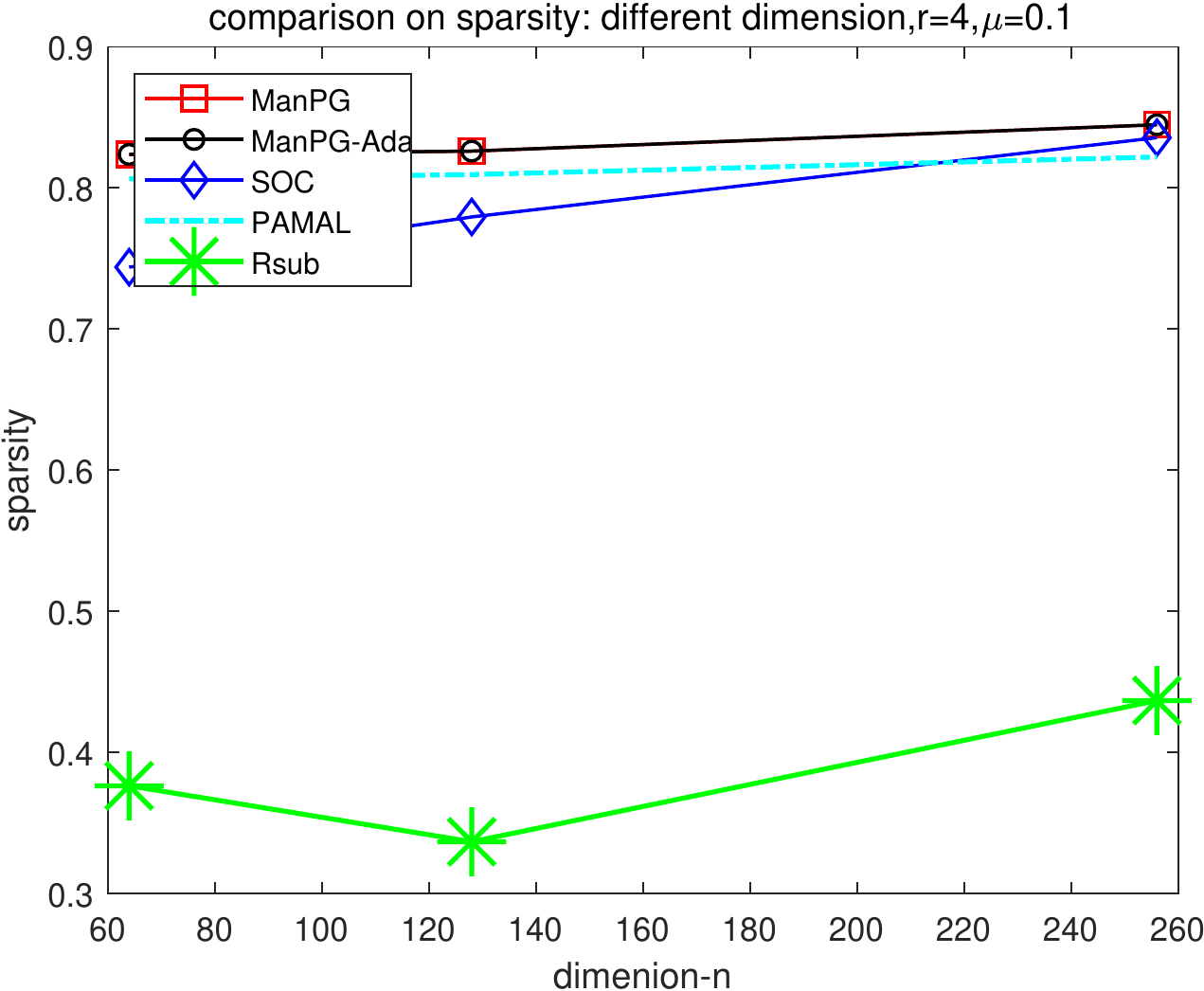}}
		\endminipage\hfill
		\caption{Comparison on CM problem \eqref{CM}, different dimension $n=\{64,128,256\}$ with $r=4$ and $\mu=0.1$.}
		\label{figure:CMS_dim_sub}
	\end{center}
	\vskip -0.1in
\end{figure}

\begin{figure}[H]
	\begin{center}
		\minipage{0.33\textwidth}
		\subfigure[CPU ]{
			\centerline{\includegraphics[width=0.9\linewidth]{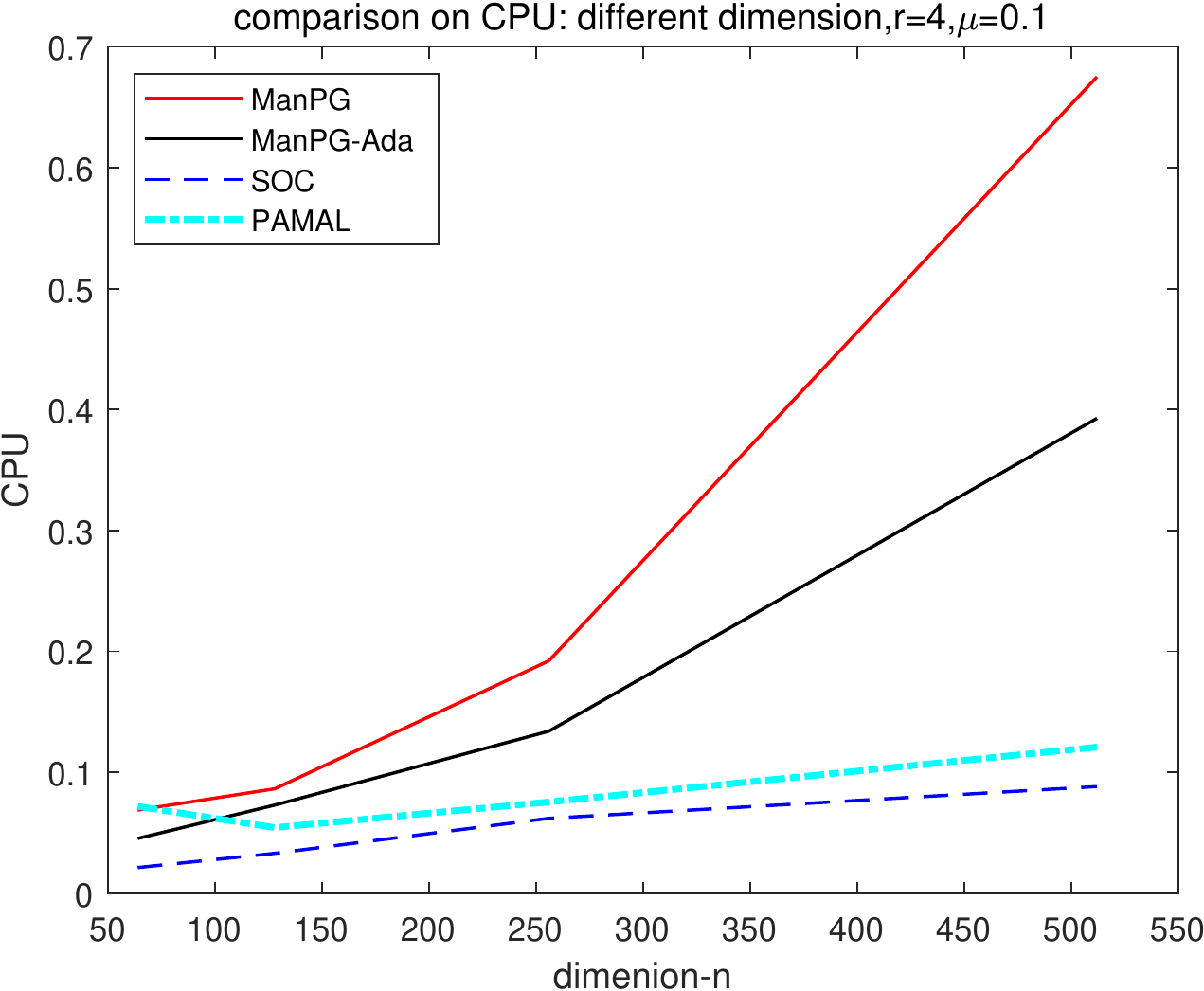}}}
		\endminipage\hfill
		\minipage{0.33\textwidth}
		\subfigure[ Iteration ]{\includegraphics[width=0.9\linewidth]{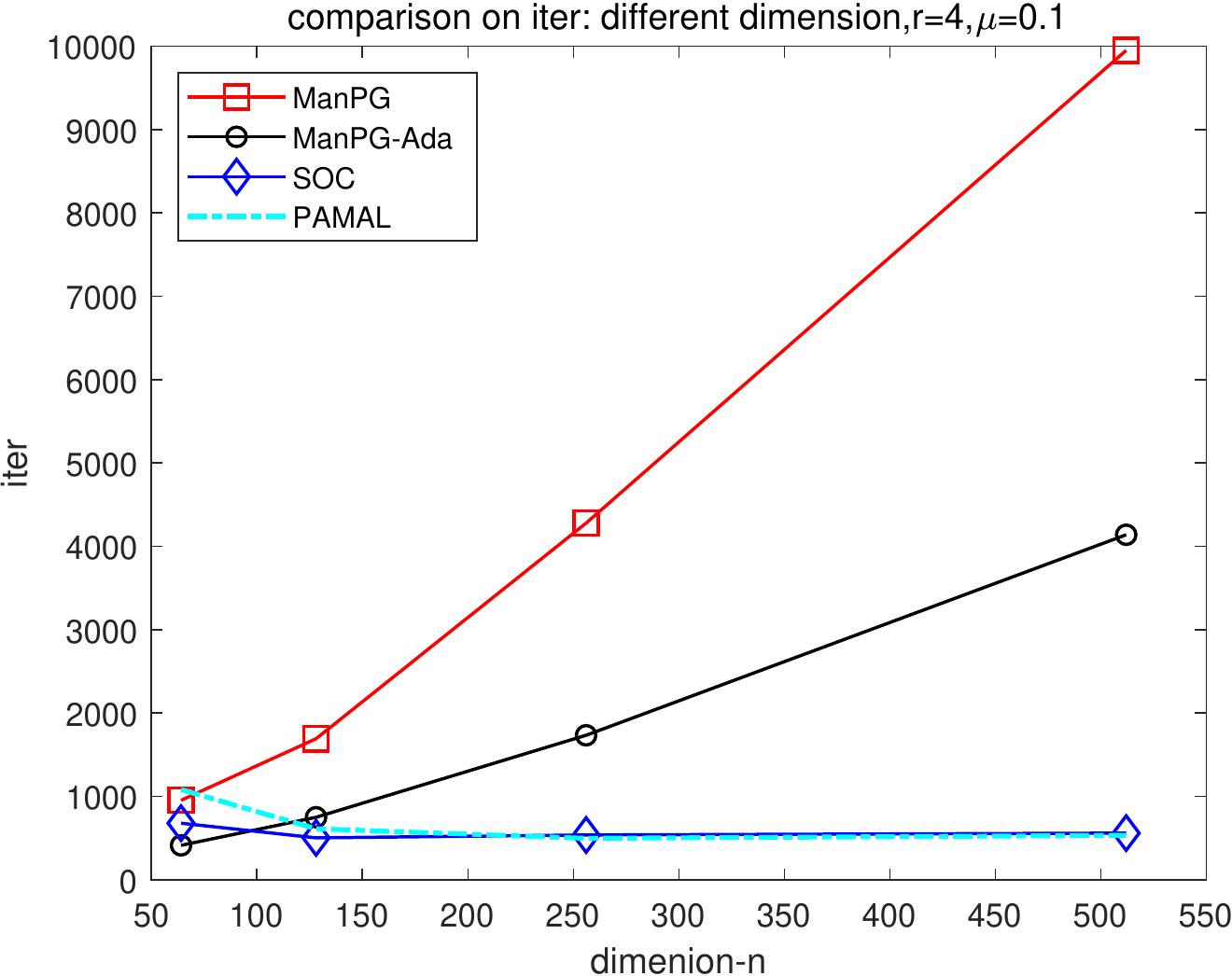}}
		\endminipage\hfill
		\minipage{0.33\textwidth}
		\subfigure[Sparsity ]{\includegraphics[width=0.9\linewidth]{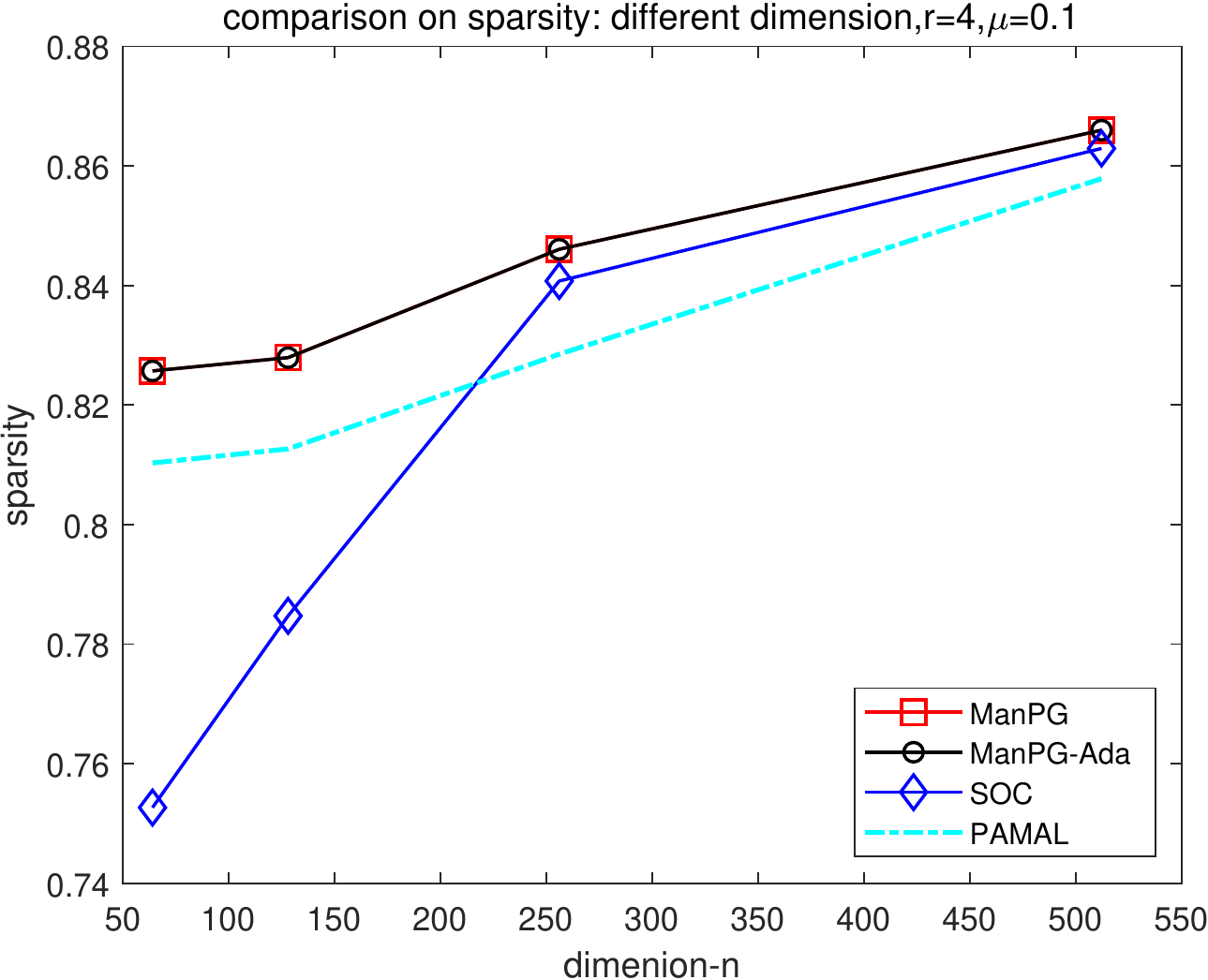}}
		\endminipage\hfill
		\caption{Comparison on CM problem \eqref{CM}, different dimension $n=\{64,128,256,512\}$ with $r=4$ and $\mu=0.1$.}
		\label{figure:CMS_dim}
	\end{center}
	\vskip -0.1in
\end{figure}

\begin{figure}[H]
	\begin{center}
		\minipage{0.33\textwidth}
		\subfigure[CPU ]{
			\centerline{\includegraphics[width=0.9\linewidth]{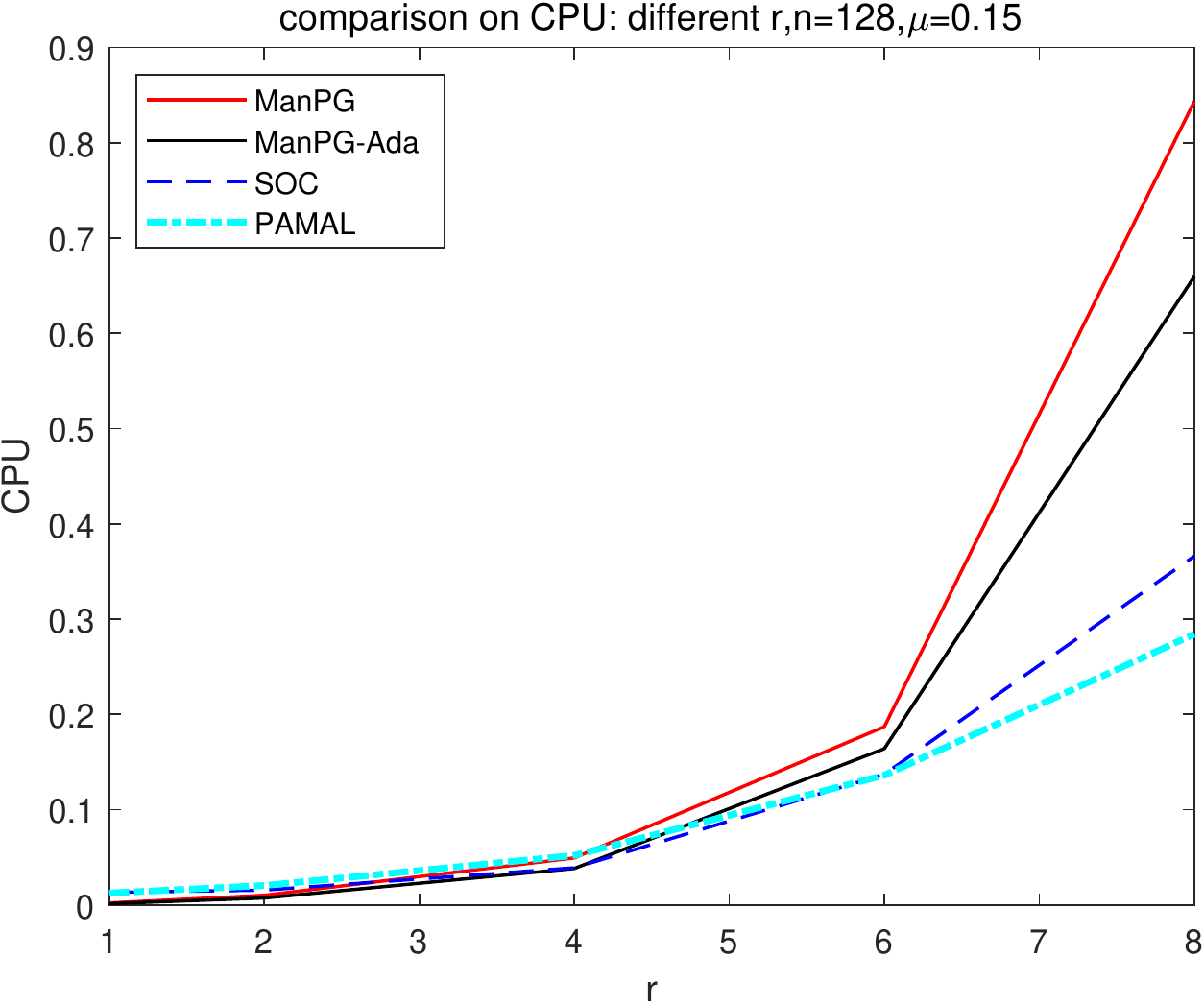}}}
		\endminipage\hfill
		\minipage{0.33\textwidth}
		\subfigure[ Iteration ]{\includegraphics[width=0.9\linewidth]{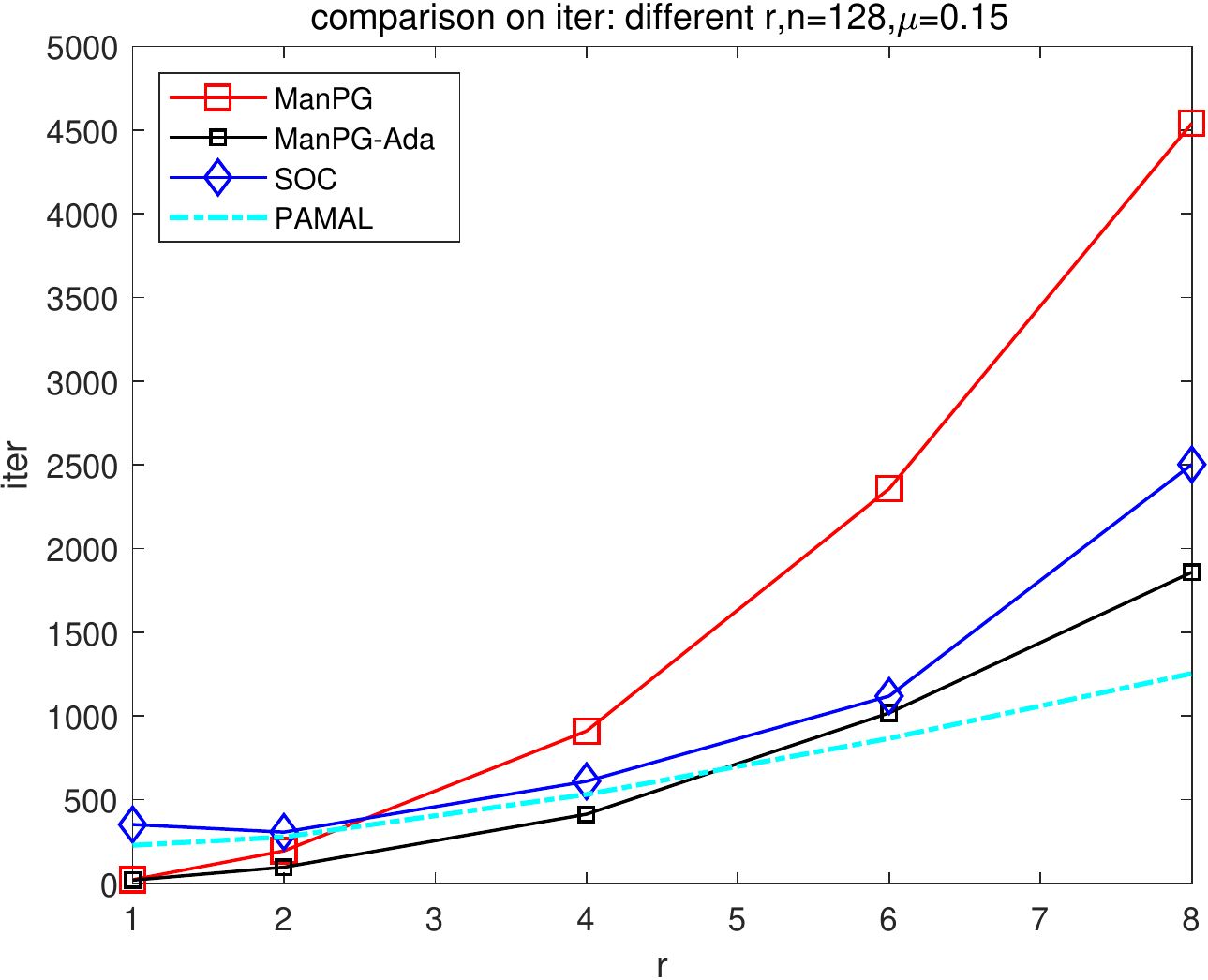}}
		\endminipage\hfill
		\minipage{0.33\textwidth}
		\subfigure[Sparsity ]{\includegraphics[width=0.9\linewidth]{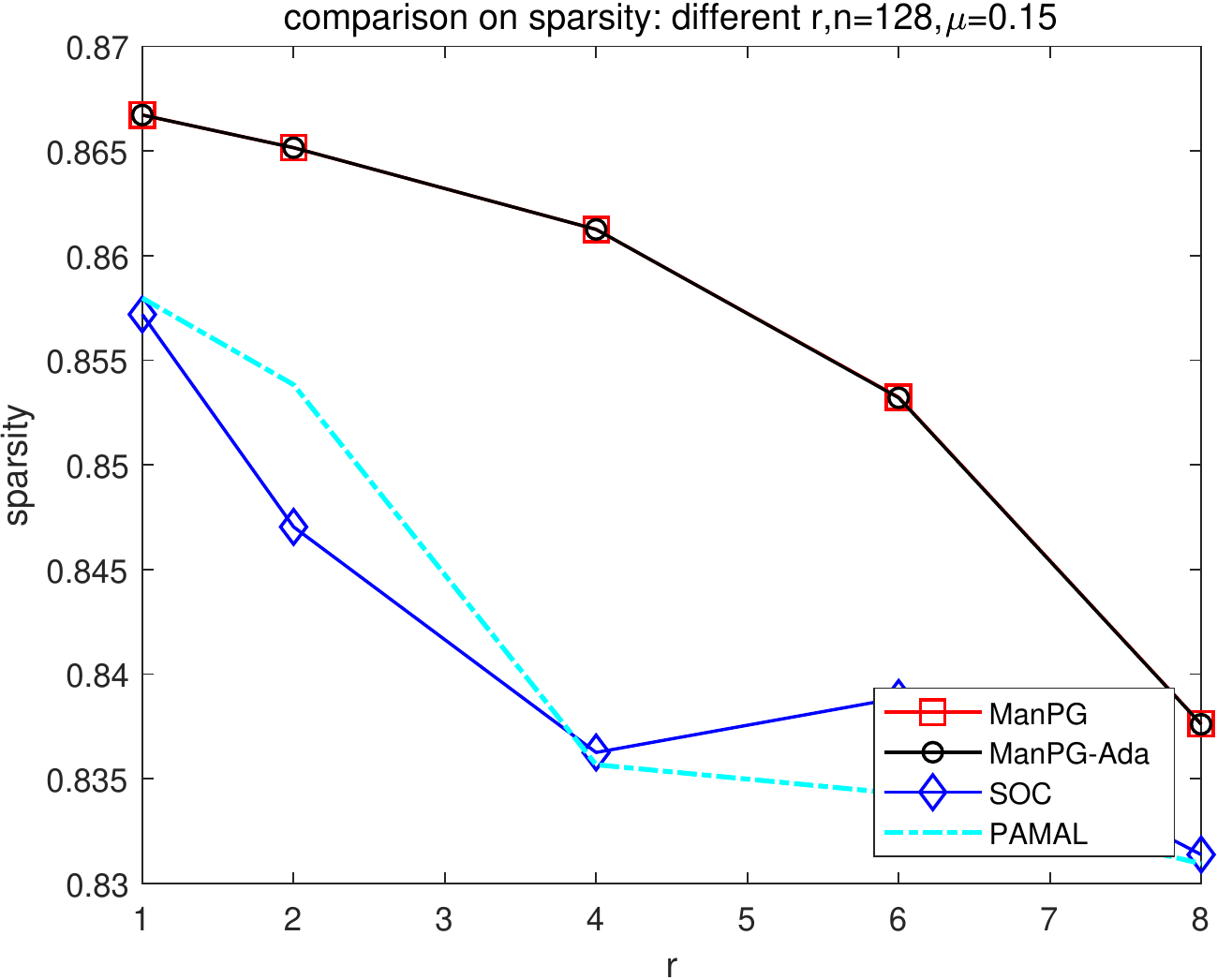}}
		\endminipage\hfill
		\caption{Comparison on CM problem \eqref{CM}, different  $r=\{1,2,4,6,8\}$ with $n=128$ and $\mu=0.15$.}
		\label{figure:CMS_rank}
	\end{center}
	\vskip -0.1in
\end{figure}

\begin{figure}[ht]
	\begin{center}
		\minipage{0.33\textwidth}
		\subfigure[CPU ]{
			\centerline{\includegraphics[width=0.9\linewidth]{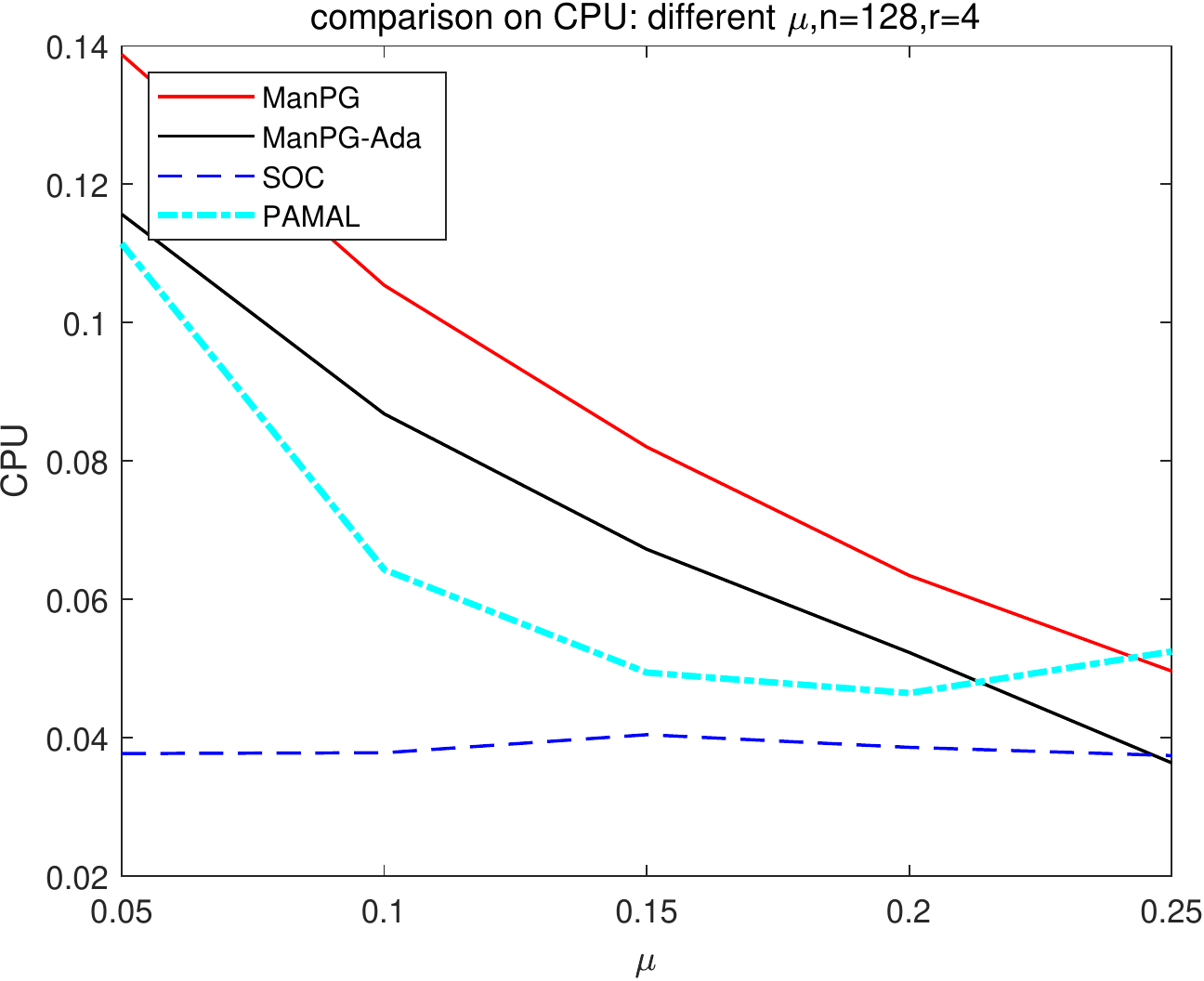}}}
		\endminipage\hfill
		\minipage{0.33\textwidth}
		\subfigure[ Iteration ]{\includegraphics[width=0.9\linewidth]{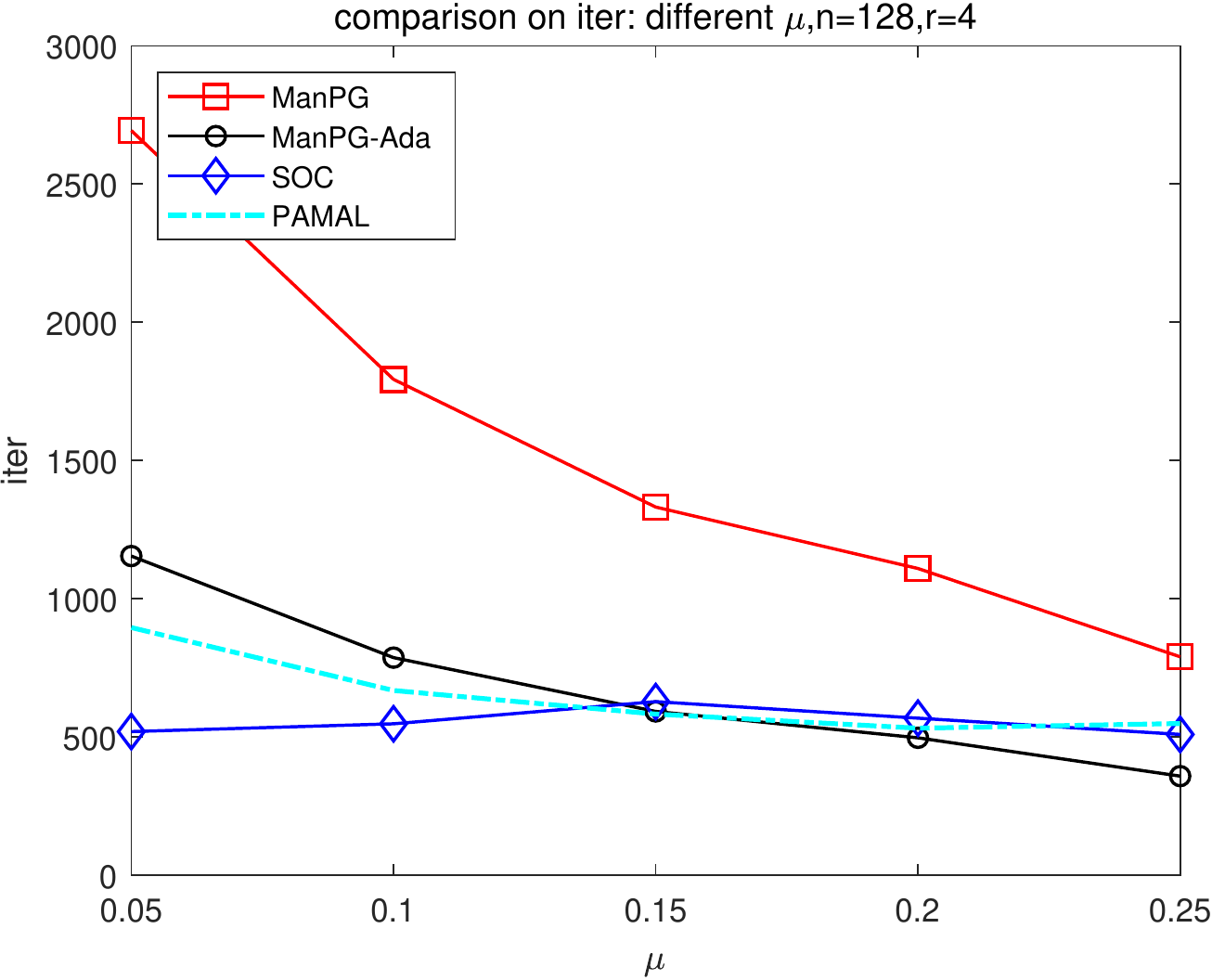}}
		\endminipage\hfill
		\minipage{0.33\textwidth}
		\subfigure[Sparsity ]{\includegraphics[width=0.9\linewidth]{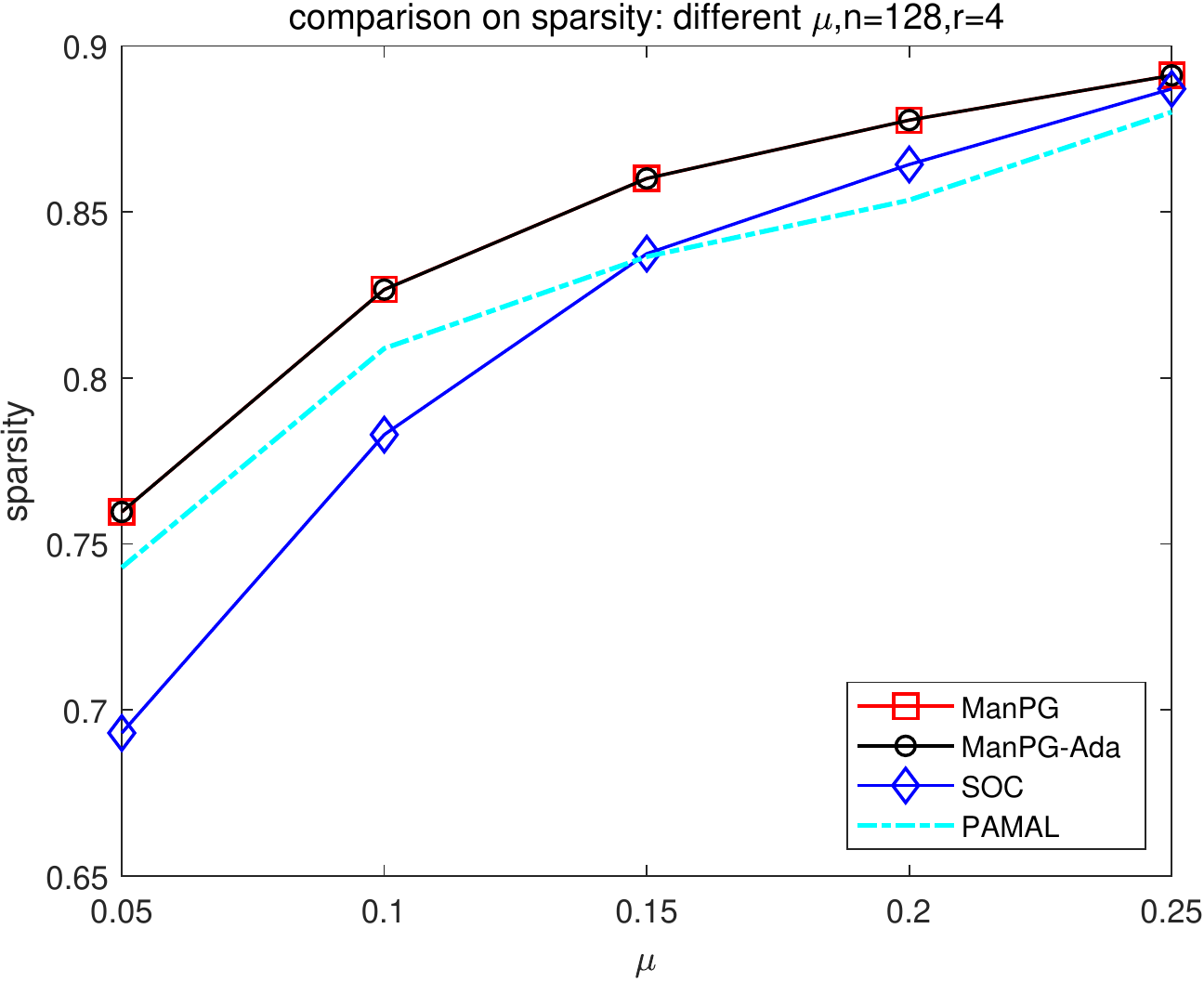}}
		\endminipage\hfill
		\caption{Comparison on CM problem \eqref{CM}, different  $\mu=\{0.05,0.1,0.15,0.2,0.25 \}$ with $n=128$ and $r=4$.}
		\label{figure:CMS_sp_mu}
	\end{center}
	\vskip -0.1in
\end{figure}

\begin{table}[H]\small
	\centering
	\caption{Number of line search steps and averaged SSN iterations for different $(n,r,\mu)$.}
	\begin{tabular}{c|cc|cc|cc}
		\hline
		\hline
		\multicolumn{1}{c}{} & \multicolumn{2}{c}{ManPG} & \multicolumn{2}{c}{ManPG-Ada} & SOC   & PAMAL \bigstrut\\
		\hline
		\multicolumn{1}{c}{} & \# line search  & \multicolumn{1}{c}{SSN iter} & \# line search  & \multicolumn{1}{c}{SSN iter} & \# s$|$ d$|$ f & \# s$|$ d$|$ f \bigstrut\\
		\hline
		\multicolumn{1}{c}{$n$} & \multicolumn{6}{c}{$r=4,\mu=0.1$} \bigstrut\\
		\hline
		64    & 85.94 & 1.0005 & 165.98 & 1.3307 & 50$|$0$|$0 & 48$|$2$|$0 \bigstrut[t]\\
		128   & 70.5  & 0.64414 & 540.76 & 1.2237 & 50$|$0$|$0 & 50$|$0$|$0 \\
		256   & 84.06 & 0.39686 & 1191.5 & 0.60652 & 50$|$0$|$0 & 50$|$0$|$0 \\
		512   & 55.1  & 0.16622 & 2720.6 & 0.2417 & 50$|$0$|$0 & 49$|$1$|$0 \bigstrut[b]\\
		\hline
		\multicolumn{1}{c}{$\mu$} & \multicolumn{6}{c}{$n =128 , r = 4$} \bigstrut\\
		\hline
		0.05  & 49.2  & 0.30933 & 695.6 & 0.83637 & 50$|$0$|$0 & 50$|$0$|$0 \bigstrut[t]\\
		0.1   & 74.38 & 0.54915 & 572.42 & 1.1514 & 50$|$0$|$0 & 50$|$0$|$0 \\
		0.15  & 102.62 & 0.82093 & 439.6 & 1.2899 & 50$|$0$|$0 & 50$|$0$|$0 \\
		0.2   & 82.52 & 0.81565 & 350.86 & 1.2114 & 50$|$0$|$0 & 50$|$0$|$0 \\
		0.25  & 93.3  & 0.57232 & 209.12 & 1.0122 & 50$|$0$|$0 & 48$|$2$|$0 \bigstrut[b]\\
		\hline
		\multicolumn{1}{c}{$r$} & \multicolumn{6}{c}{$n=128, \mu=0.15$} \bigstrut\\
		\hline
		1     & 0     & 0.8971 & 0     & 0.98694 & 50$|$0$|$0 & 50$|$0$|$0 \bigstrut[t]\\
		2     & 3.48  & 1.0001 & 61.02 & 1.1135 & 50$|$0$|$0 & 50$|$0$|$0 \\
		4     & 86.92 & 0.91814 & 311   & 1.2812 & 50$|$0$|$0 & 50$|$0$|$0 \\
		6     & 169.8 & 0.60206 & 719.42 & 1.5195 & 50$|$0$|$0 & 49$|$1$|$0 \\
		8     & 216.54 & 1.2011 & 1198.8 & 2.8667 & 50$|$0$|$0 & 42$|$8$|$0 \bigstrut[b]\\
		\hline
	\end{tabular}%
	\label{tab:CMs_ave_line_SSN}%
\end{table}%

\subsection{Numerical Results on Sparse PCA}

In this section, we compare the performance of ManPG, ManPG-Ada, SOC, and PAMAL for solving the sparse PCA problem \eqref{spca}. Note that there are other algorithms for sparse PCA such as the ones proposed in \cite{Jolliffe2003,daspremont-sparsePCA-direct-formulation-2007}, but these methods work only for the special case when $r=1$; i.e., the constraint set is a sphere. The algorithm proposed in \cite{Genicot-weakly-2015} needs to smooth the $\ell_1$ norm in order to apply existing gradient-type methods and thus the sparsity of the solution is no longer guaranteed. Algorithms proposed in \cite{Zou-spca-2006,Shen-Huang-spca-2008,Journee-Nesterov-sparsePCA-JMLR-2010} do not impose orthogonal loading directions. In other words, they cannot impose both sparsity and orthogonality on the same variable. Therefore, we chose not to compare our ManPG with these algorithms.

The random data matrices $A\in \R^{m\times n}$ considered in this section were generated in the following way. We first generate a random matrix using the MATLAB function $A=randn(m,n)$, then shift the columns of $A$ so that their mean is equal to $0$, and lastly normalize the columns so that their Euclidean norms are equal to one. In all tests, $m$ is equal to 50. The Lipschitz constant $L$ is $2\sigma^2_{\max}(A)$, so we use $t =1/(2\sigma^2_{\max}(A))$ in Algorithms \ref{alg:ManPG} and \ref{alg:ManPG-adap}, where $\sigma_{\max}(A)$ is the largest singular value of $A$. Again, we spent a lot of effort on tuning the parameters for SOC and PAMAL and found that the following settings of the parameters worked best for our tested problems. For SOC, we set the penalty parameters $\beta = 2\sigma^2_{\max}(A)$. For PAMAL, we set $\tau=0.99$, $\gamma=1.001$, $\rho^{1}= 5\sigma^2_{\max}(A)$, $\overline{\Lambda}_{p, \min }=-100$, $\overline{\Lambda}_{p, \max }=100$, $\Lambda_{p}^{1}=0_{nr}, p=1,2,$ and
$\epsilon^{k}=(0.996)^{k}$, $k \in \mathbb{N}$. We again refer the reader to page B587 of \cite{Chen2016} for the meanings of these parameters. We used the same parameters of PAM in PAMAL as suggested in \cite{Chen2016}. 
We used the same stopping criterion for ManPG, ManPG-Ada, SOC, and PAMAL as for the CM problems. For different settings of $(n,r,\mu)$, we ran the four algorithms with $50$ instances whose initial points were obtained by projecting randomly generated points onto $\St(n,r)$. We then ran the  Riemannian subgradient method \eqref{alg:rieman_subgrad} for 500 iterations and used the returned solution to be the initial point of the compared solvers. 

The CPU time, iteration number, and sparsity are reported in Figures \ref{figure:SPCA_dim}, \ref{figure:SPCA_rank}, and \ref{figure:SPCA_sp_mu}. Same as the CM problem, all the values were averaged over those instances that yielded solutions that were close to the ones given by ManPG. 
In Figures \ref{figure:SPCA_dim}, \ref{figure:SPCA_rank} and \ref{figure:SPCA_sp_mu}, we see that ManPG and ManPG-Ada significantly outperformed SOC and PAMAL in terms of CPU time required to obtain the same solutions. 
We also see that ManPG-Ada greatly improved the performance of ManPG. We also report the total number of line search steps and the averaged iteration number of SSN in ManPG and ManPG-Ada in Table \ref{tab:SPCA_ave_line_SSN}. We observe from Table \ref{tab:SPCA_ave_line_SSN} that SOC failed to converge on one instance, and for several instances, SOC and PAMAL generated different solutions when compared to those generated by ManPG. 

\begin{figure}[H]
	\begin{center}
		\minipage{0.33\textwidth}
		\subfigure[CPU ]{
			\centerline{\includegraphics[width=0.9\linewidth]{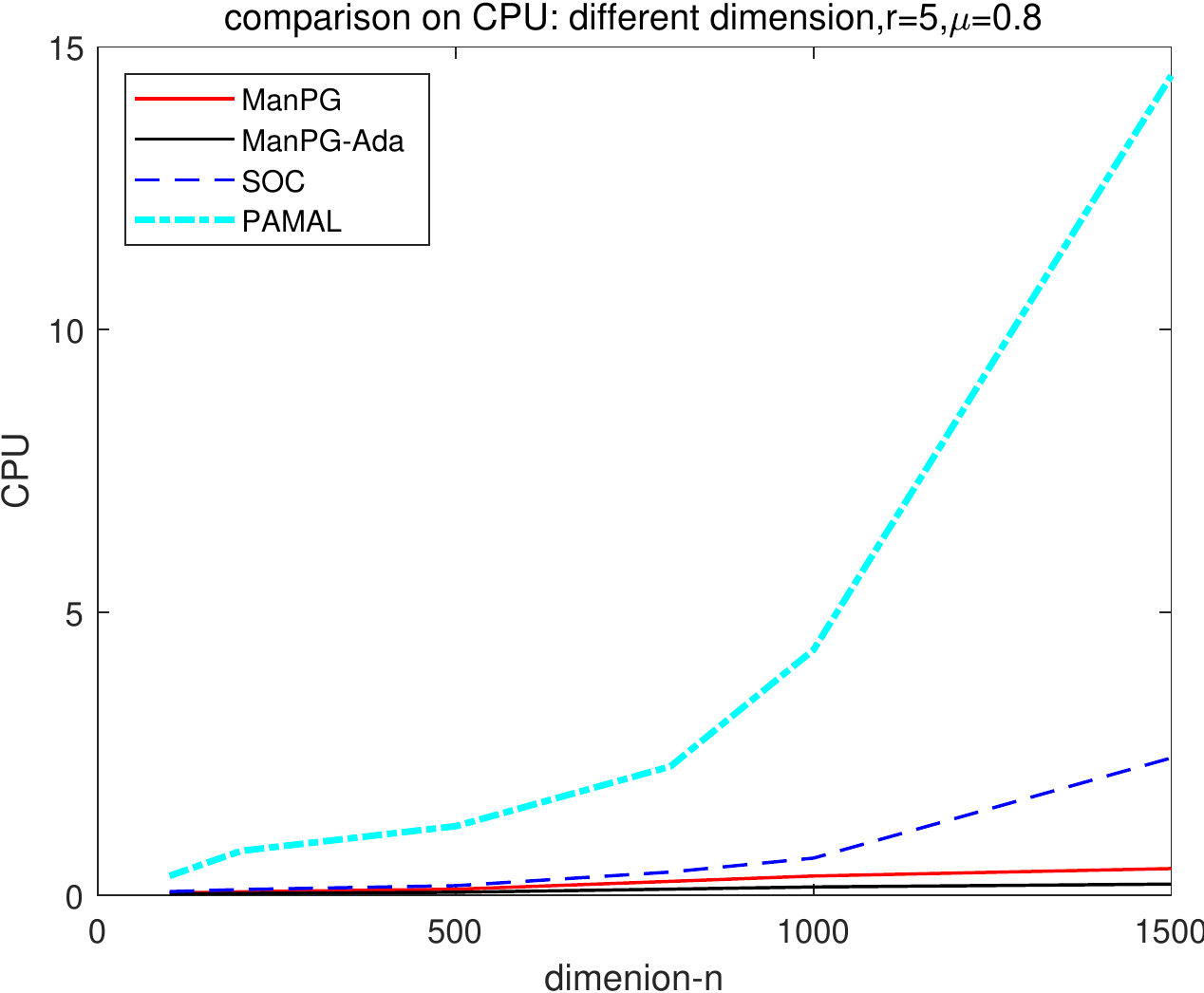}}}
		\endminipage\hfill
		\minipage{0.33\textwidth}
		\subfigure[ Iteration ]{\includegraphics[width=0.9\linewidth]{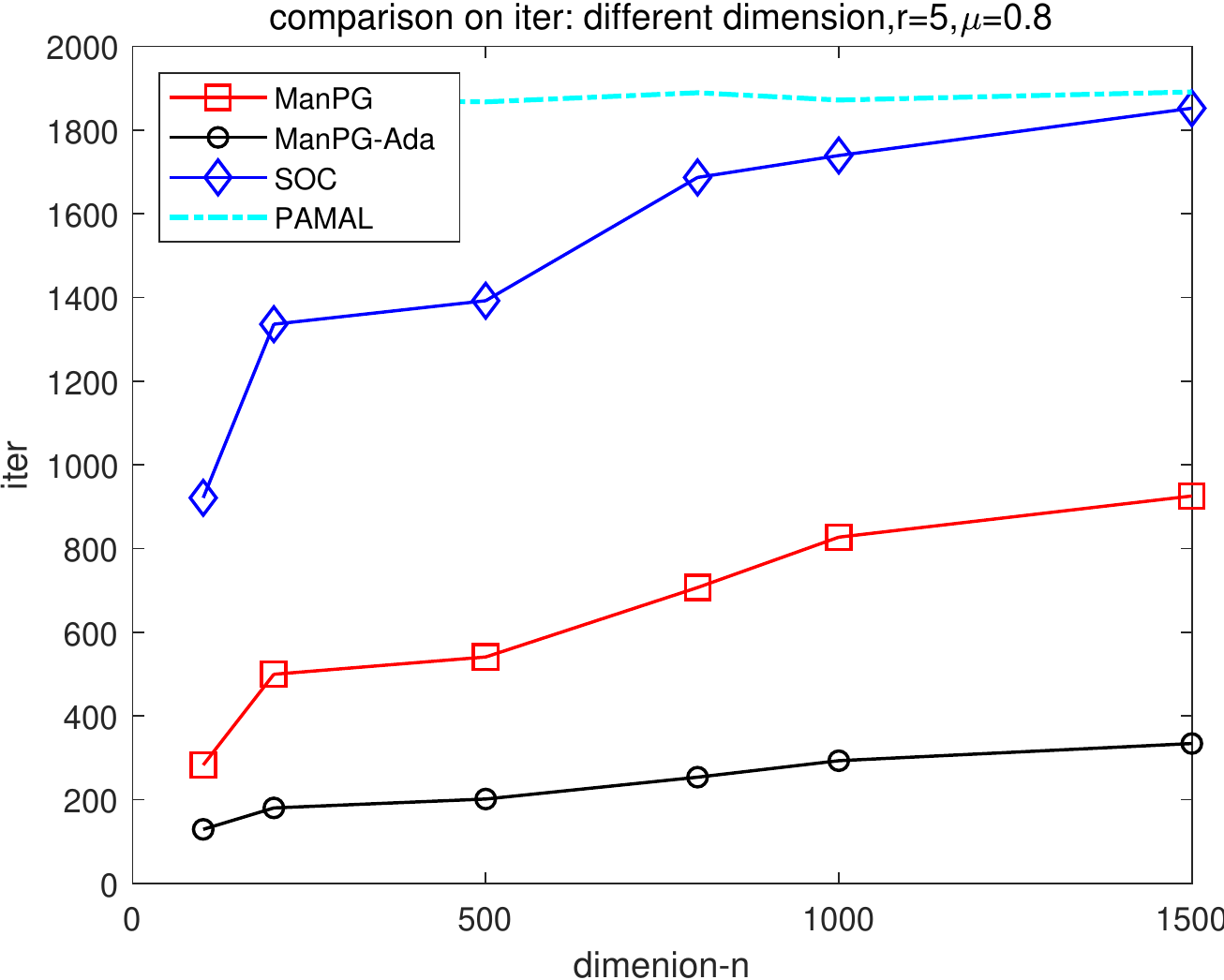}}
		\endminipage\hfill
		\minipage{0.33\textwidth}
		\subfigure[Sparsity ]{\includegraphics[width=0.9\linewidth]{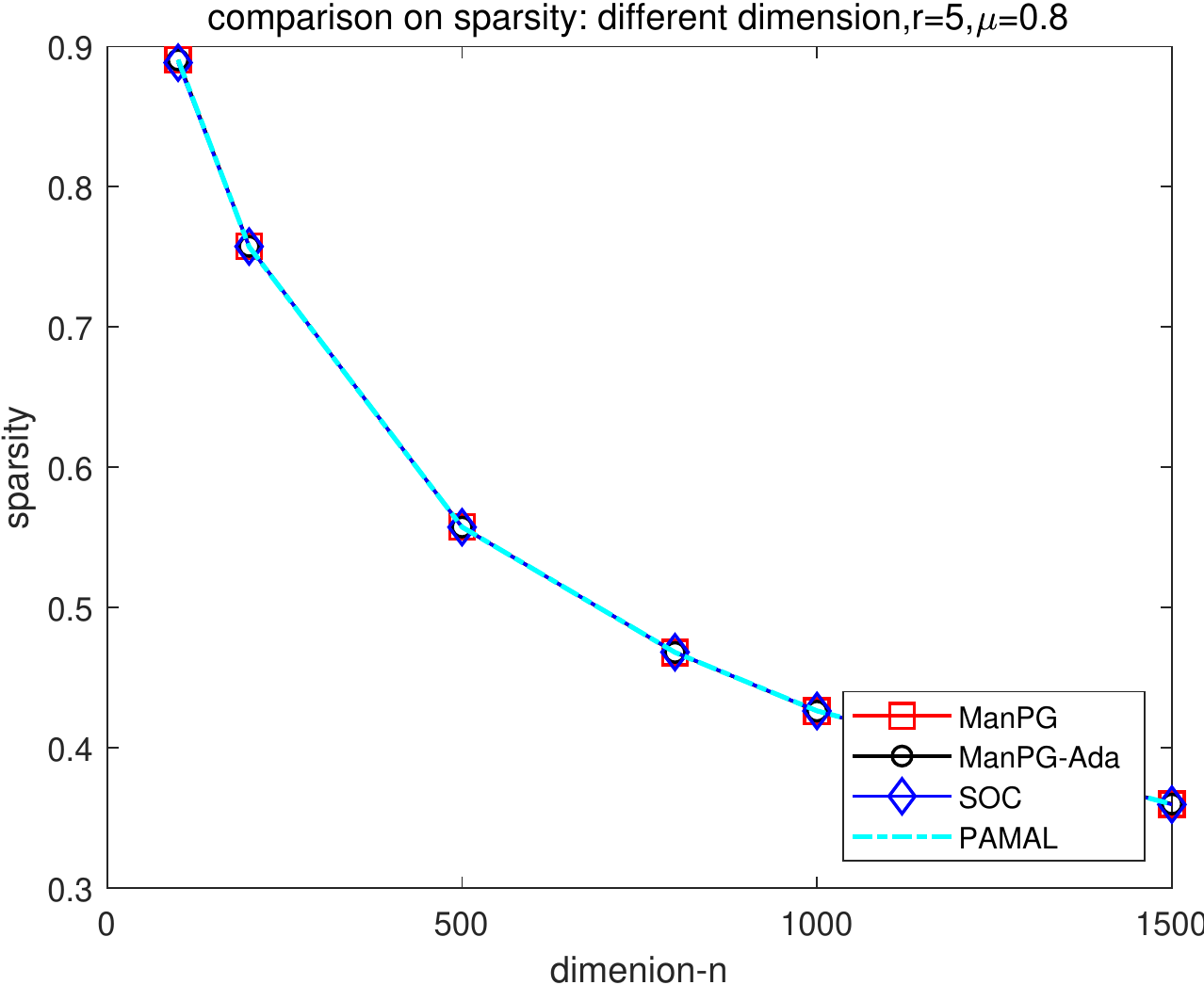}}
		\endminipage\hfill
		\caption{Comparison on SPCA problem \eqref{spca}, different dimension $n=\{100,200,500,800,1000,1500\}$ with $r=5$ and $\mu=0.8$.}
		\label{figure:SPCA_dim}
	\end{center}
	\vskip -0.1in
\end{figure}

\begin{figure}[H]
	\begin{center}
		\minipage{0.33\textwidth}
		\subfigure[CPU ]{
			\centerline{\includegraphics[width=0.9\linewidth]{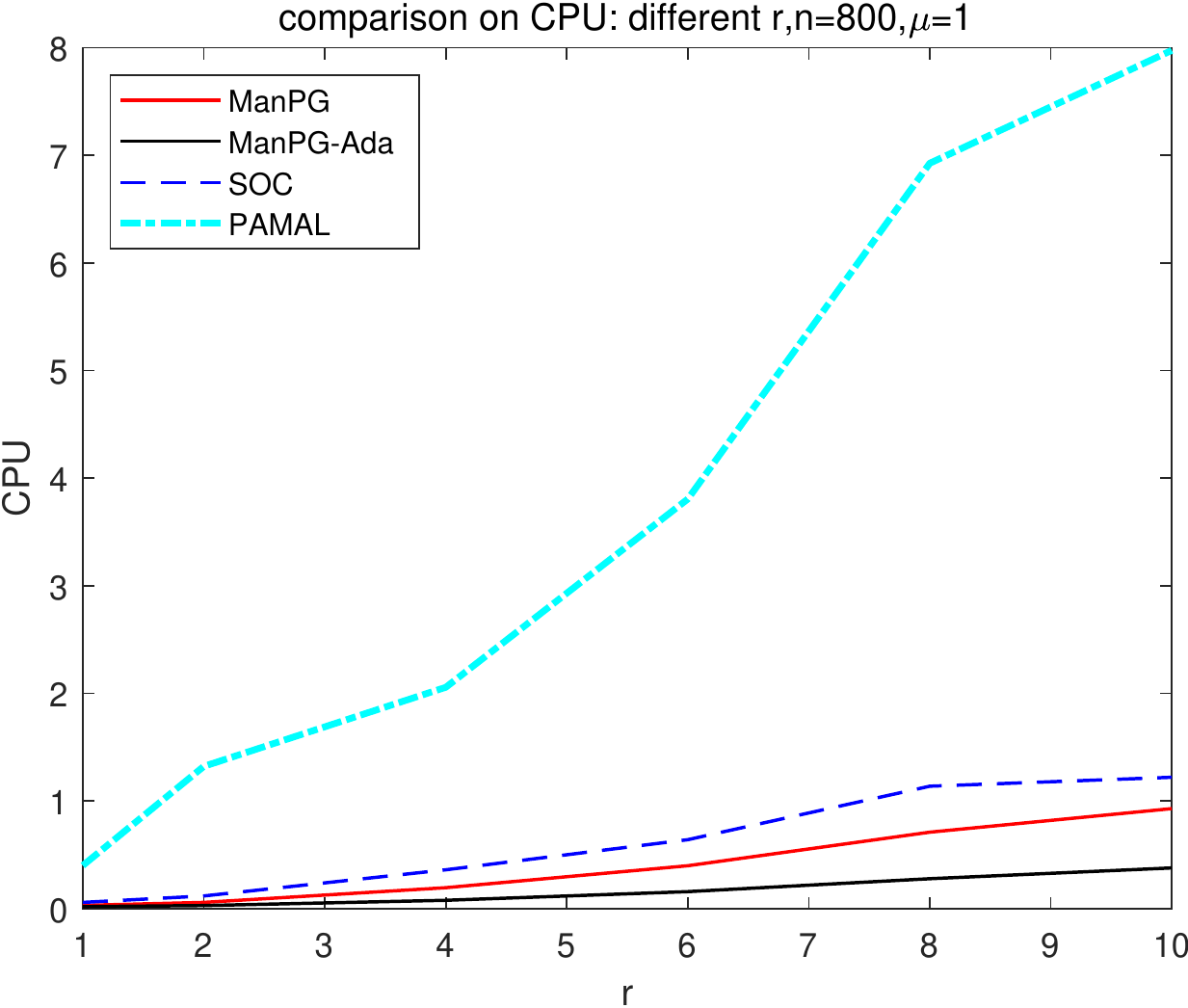}}}
		\endminipage\hfill
		\minipage{0.33\textwidth}
		\subfigure[ Iteration ]{\includegraphics[width=0.9\linewidth]{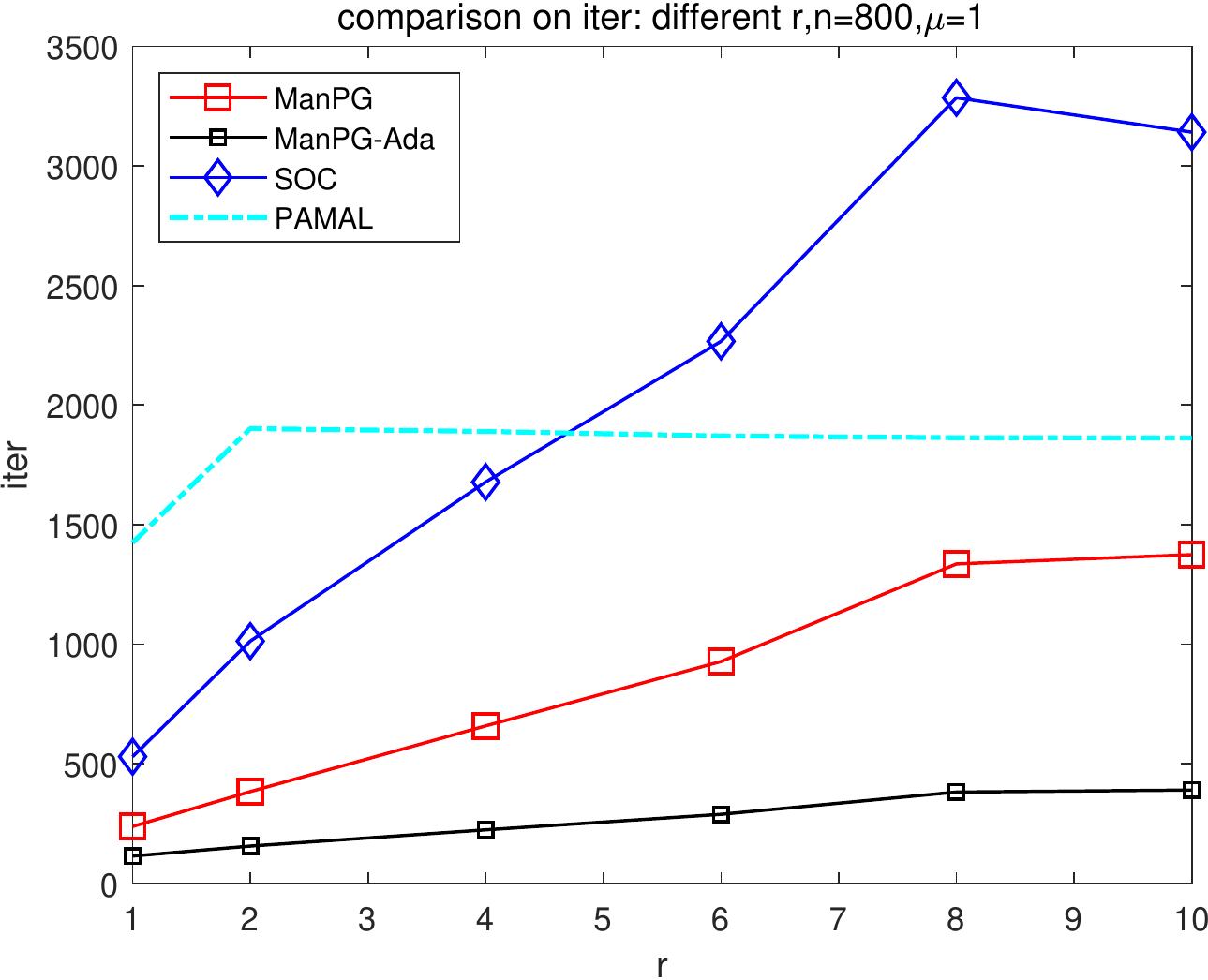}}
		\endminipage\hfill
		\minipage{0.33\textwidth}
		\subfigure[Sparsity ]{\includegraphics[width=0.9\linewidth]{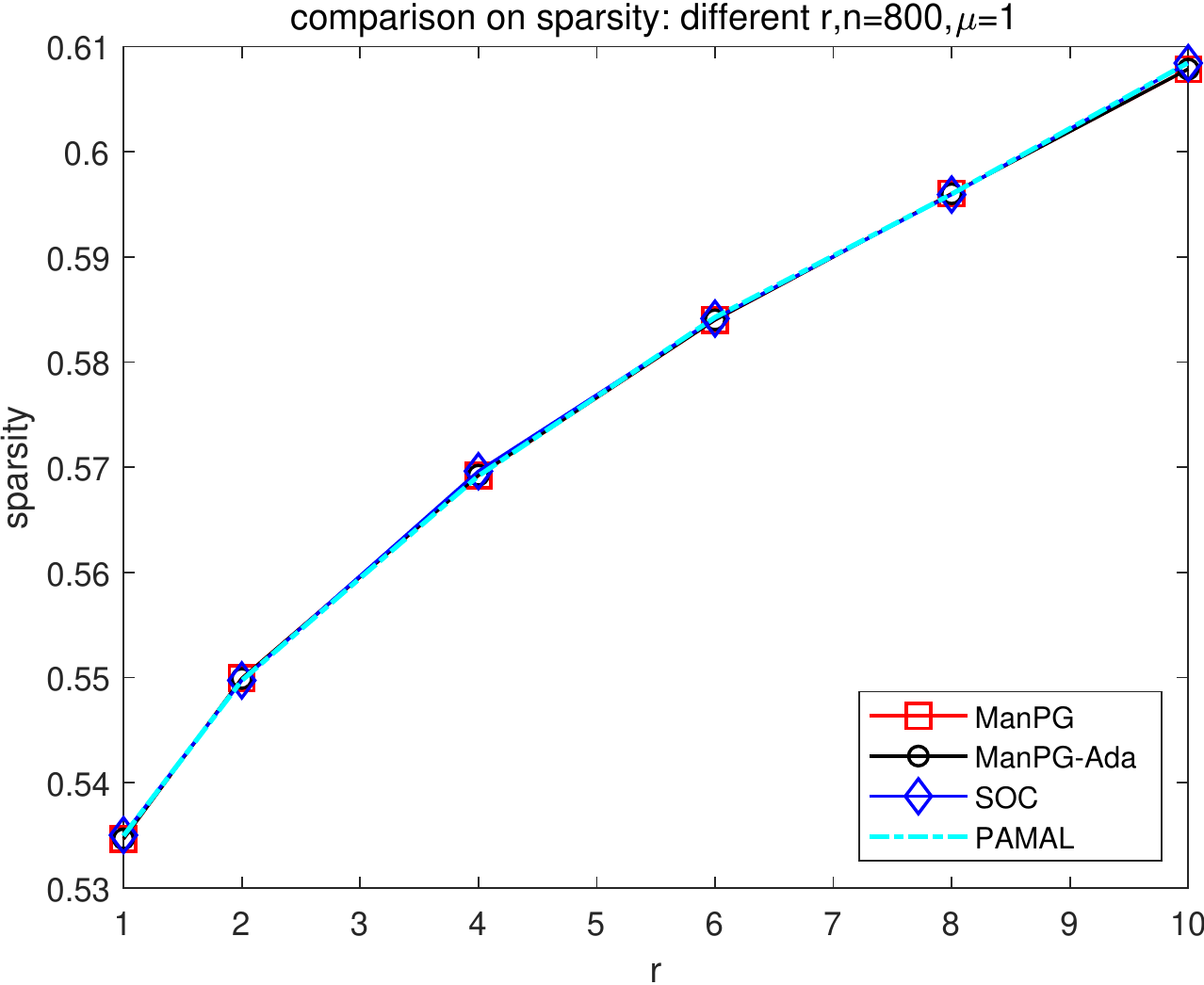}}
		\endminipage\hfill
		\caption{Comparison on SPCA problem \eqref{spca}, different  $r=\{1,2,4,6,8,10\}$ with $n=800$ and $\mu=1$.}
		\label{figure:SPCA_rank}
	\end{center}
	\vskip -0.1in
\end{figure}

\begin{figure}[H]
	\begin{center}
		\minipage{0.33\textwidth}
		\subfigure[CPU ]{
			\centerline{\includegraphics[width=0.9\linewidth]{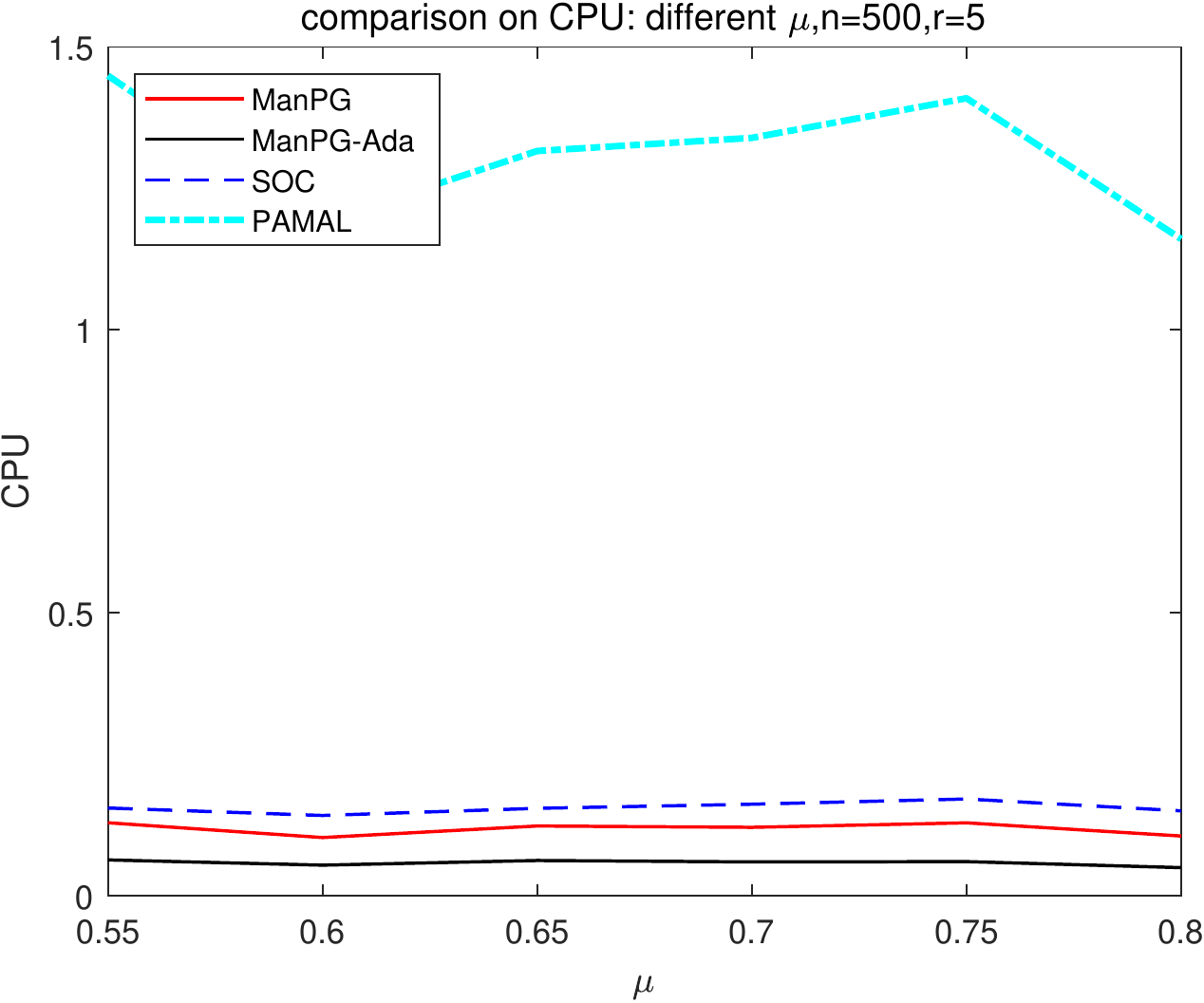}}}
		\endminipage\hfill
		\minipage{0.33\textwidth}
		\subfigure[ Iteration ]{\includegraphics[width=0.9\linewidth]{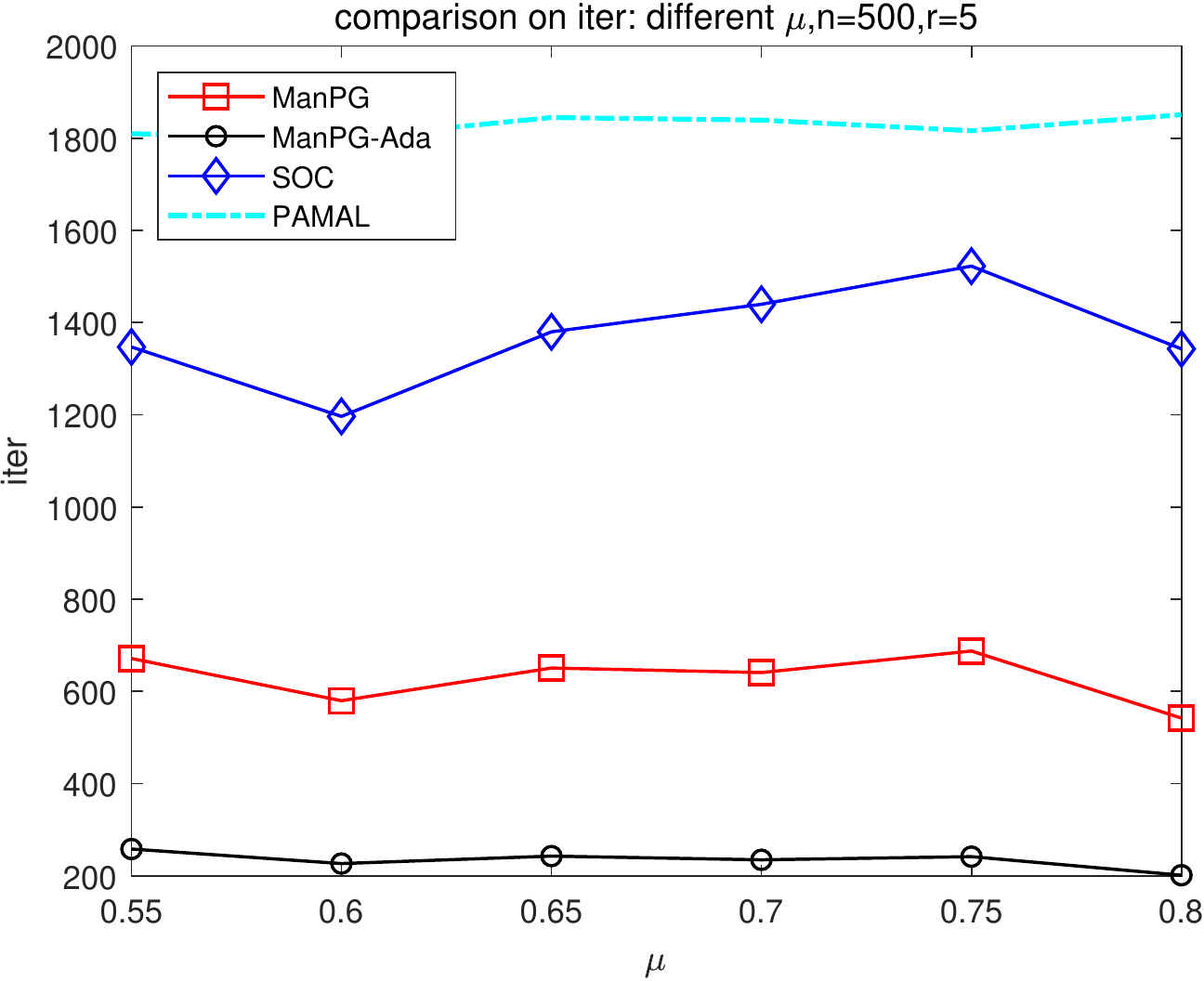}}
		\endminipage\hfill
		\minipage{0.33\textwidth}
		\subfigure[Sparsity ]{\includegraphics[width=0.9\linewidth]{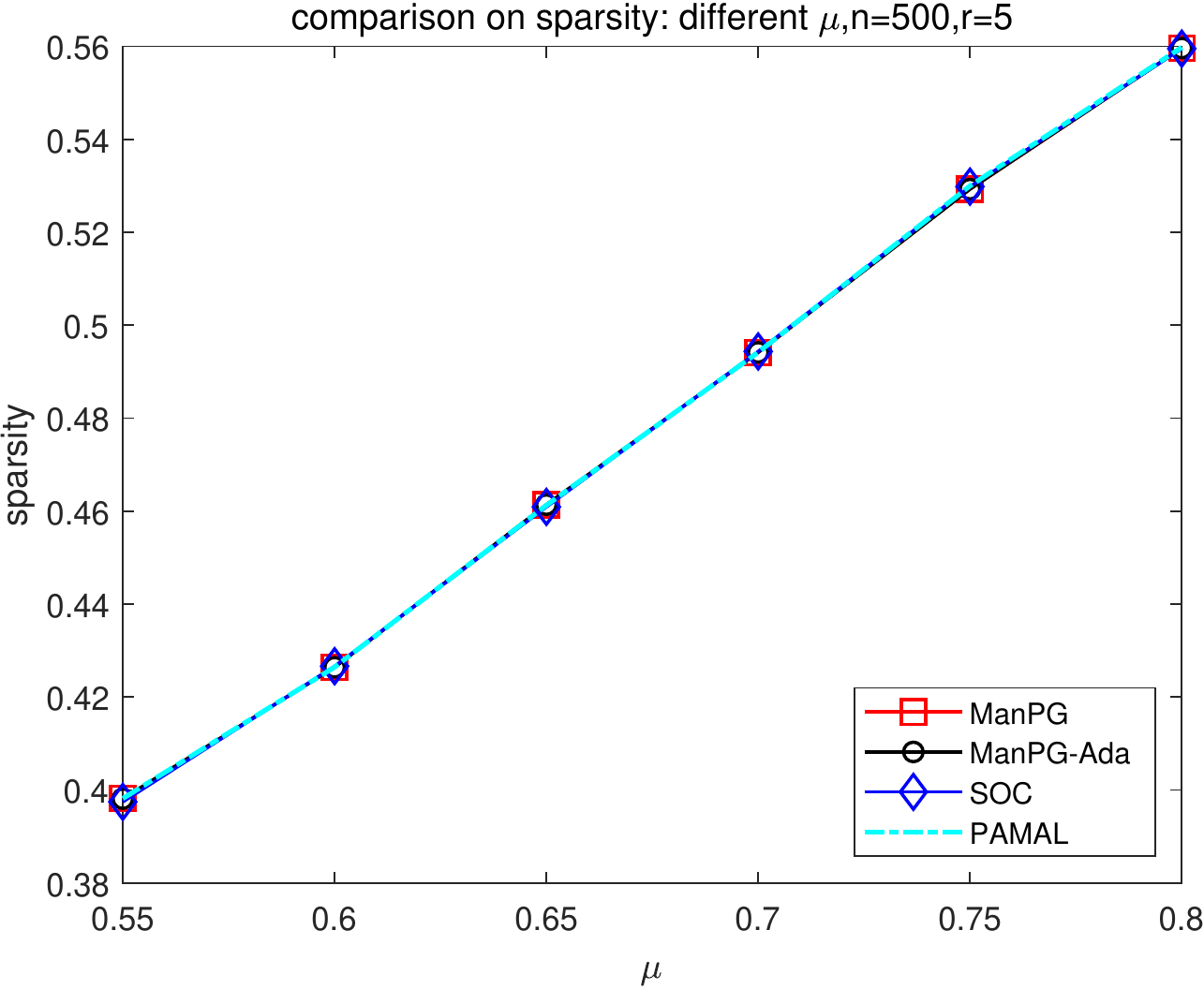}}
		\endminipage\hfill
		\caption{Comparison on SPCA problem \eqref{spca}, different  $\mu=\{0.55,0.6,0.65,0.7,0.75,0.8 \}$ with $n=500$ and $r=5$.}
		\label{figure:SPCA_sp_mu}
	\end{center}
	\vskip -0.1in
\end{figure}
\begin{table}[htbp]\small
	\centering
	\caption{SPCA: Number of line search steps and averaged SSN iterations for different $(n,r,\mu)$.}
	\begin{tabular}{c|cc|cc|cc}
		\hline
		\hline
		\multicolumn{1}{c}{} & \multicolumn{2}{c}{ManPG} & \multicolumn{2}{c}{ManPG-Ada} & SOC   & PAMAL \bigstrut\\
		\hline
		\multicolumn{1}{c}{} & \# line search  & \multicolumn{1}{c}{SSN iter} & \# line search  & \multicolumn{1}{c}{SSN iter} & \# s$|$ d$|$ f & \# s$|$ d$|$ f \bigstrut\\
		\hline
		\multicolumn{1}{c}{$n$} & \multicolumn{6}{c}{$r=5,\mu=0.8$} \bigstrut\\
		\hline
		100   & 0.8   & 1.1881 & 0.08  & 1.5221 & 46$|$3$|$1 & 50$|$0$|$0 \bigstrut[t]\\
		200   & 2.98  & 1.0722 & 15.1  & 1.3705 & 48$|$2$|$0 & 48$|$2$|$0 \\
		500   & 0.4   & 1.025 & 29.4  & 1.2066 & 50$|$0$|$0 & 50$|$0$|$0 \\
		800   & 0     & 1.0167 & 59.36 & 1.1847 & 49$|$1$|$0 & 50$|$0$|$0 \\
		1000  & 3.08  & 1.016 & 82.04 & 1.1712 & 49$|$1$|$0 & 49$|$1$|$0 \\
		15    & 11    & 1.0121 & 108.94 & 1.1035 & 48$|$2$|$0 & 49$|$1$|$0 \bigstrut[b]\\
		\hline
		\multicolumn{1}{c}{$\mu$} & \multicolumn{6}{c}{$n =500 , r = 5$} \bigstrut\\
		\hline
		0.55  & 0     & 1.0155 & 68.7  & 1.1463 & 48$|$2$|$0 & 50$|$0$|$0 \bigstrut[t]\\
		0.60  & 0     & 1.0197 & 48.82 & 1.1431 & 50$|$0$|$0 & 49$|$1$|$0 \\
		0.65  & 0     & 1.019 & 57.96 & 1.1841 & 48$|$2$|$0 & 48$|$2$|$0 \\
		0.70  & 0     & 1.0246 & 52.5  & 1.2098 & 49$|$1$|$0 & 50$|$0$|$0 \\
		0.75  & 0.36  & 1.0238 & 55.88 & 1.2252 & 48$|$2$|$0 & 49$|$1$|$0 \\
		0.80  & 0     & 1.0286 & 28.98 & 1.1966 & 49$|$1$|$0 & 49$|$1$|$0 \bigstrut[b]\\
		\hline
		\multicolumn{1}{c}{$r$} & \multicolumn{6}{c}{$n=800, \mu=0.6$} \bigstrut\\
		\hline
		1     & 0     & 0.90182 & 4.12  & 1.0335 & 50$|$0$|$0 & 50$|$0$|$0 \bigstrut[t]\\
		2     & 82.06 & 1.0041 & 10.74 & 1.0767 & 49$|$1$|$0 & 50$|$0$|$0 \\
		4     & 8.52  & 1.0229 & 39.04 & 1.1453 & 48$|$2$|$0 & 50$|$0$|$0 \\
		6     & 0     & 1.0243 & 72.22 & 1.3198 & 46$|$4$|$0 & 49$|$1$|$0 \\
		8     & 0.34  & 1.0309 & 125.64 & 1.5325 & 46$|$4$|$0 & 50$|$0$|$0 \\
		10    & 0.76  & 1.0579 & 132.58 & 1.6894 & 42$|$8$|$0 & 47$|$3$|$0 \bigstrut[b]\\
		\hline
	\end{tabular}%
	\label{tab:SPCA_ave_line_SSN}%
\end{table}%

\section{Discussions and Concluding Remarks}\label{sec:conclude}


Manifold optimization has attracted a lot of attention recently. In this paper, we proposed a proximal gradient method (ManPG) for solving the nonsmooth nonconvex optimization problem over the Stiefel manifold \eqref{prob-g}.
Different from existing methods, our ManPG algorithm relies on proximal gradient information on the tangent space rather than subgradient information. Under the assumption that the the smooth part of the objective function has a Lipschitz continuous gradient, we proved that ManPG converges globally to a stationary point of \eqref{prob-g}. Moreover, we analyzed the iteration complexity of ManPG for obtaining an $\epsilon$-stationary solution. Our numerical experiments suggested that when combined with a regularized semi-smooth Newton method for finding the descent direction, ManPG performs efficiently and robustly. In particular, ManPG is more robust than SOC and PAMAL for solving the compressed modes and sparse PCA problems, as it is less sensitive to the choice of parameters. Moreover, ManPG significantly outperforms SOC and PAMAL for solving the sparse PCA problem in terms of the CPU time needed for obtaining the same solution.

 {
It is worth noting that the convergence and iteration complexity analyses in Section \ref{sec:convergence} also hold for other, not necessarily bounded, embedded submanifolds of an Euclidean space, provided that the objective function $F$ satisfies some additional assumptions (e.g., $F$ is coercive and and lower bounded on $\M$). We focused on the Stiefel manifold because it is easier to discuss the semi-smooth Newton method in Section \ref{sec:ssn} for finding the descent direction. As demonstrated in our tests on the compressed modes and sparse PCA problems, the efficiency of ManPG highly relies on that of solving the convex subproblem to find the descent direction. For general embedded submanifolds, it remains an interesting question whether the matrix $\mathcal{A}_k$ in \eqref{tangent-subproblem} can be easily computed and the resulting subproblem can be solved efficiently.}




\section*{Acknowledgements}
We are grateful to the associate editor and two anonymous referees for constructive and insightful comments that significantly improve the presentation of this paper. We would like to thank Rongjie Lai for discussions on SOC and the compressed modes problem, Lingzhou Xue for discussions on the sparse PCA problem, Zaiwen Wen for discussions on semi-smooth Newton methods, and Wen Huang and Ke Wei for discussions on manifolds.

\appendix

\section{Semi-smoothness of Proximal Mapping} \label{sec:semi-smooth}

\begin{definition} \label{def:clarke}
	Let $E:\Omega\rightarrow \R^q$ be locally Lipschitz continuous at $X\in \Omega\subset \R^p$. The $B$-subdifferential of $E$ at $X$ is defined by
	\[ \partial_B E(X):= \left\{ \lim_{k\rightarrow \infty} E'(X_k) \,\Big|\, X^k\in D_E, X_k \rightarrow X \right\}, \]
	where $D_E$ is the set of differentiable points of $E$ in $\Omega$. The set $\partial E(X) = \text{conv}(\partial_B E(X))$ is called Clarke's generalized Jacobian, where $\text{conv}$ denotes the convex hull.
\end{definition}
{
	Note that if $q = 1$ and $E$ is convex, then the definition is the same as that of standard convex subdifferential. Thus, we use the notation $\partial$ in Definition~\ref{def:clarke}.}
\begin{definition}\cite{Mifflin-1977,Qi-Sun-1993}
	Let $E:\Omega\rightarrow \R^q$ be locally Lipschitz continuous at $X\in \Omega\subset \R^p$. We say that $E$ is semi-smooth at $X\in \Omega$ if $E$ is directionally differentiable at $X$ and for any $J\in \partial E(X+\Delta X)$ with $\Delta X\rightarrow 0$,
	\[E(X+\Delta X) - E(X) - J\Delta X = o(\normtwo{\Delta X}). \]
	We say that $E$ is strongly semi-smooth at $X$ if $E$ is semi-smooth at $X$ and
	\[E(X+\Delta X) - E(X) - J\Delta X = O(\normtwo{\Delta X}^2). \]	
	We say that $E$ is semi-smooth on $\Omega$ if it is semi-smooth at every $X\in\Omega$.
\end{definition}
The proximal mapping of $\ell_p$ ($p\geq 1$) norm is strongly semi-smooth \cite{facchinei2007finite,ulbrich2011semismooth}. From \cite[Prop. 2.26]{ulbrich2011semismooth}, if $E: \Omega\rightarrow \R^m$ is a piecewise $\mathcal{C}^1$ (piecewise smooth) function, then $E$ is semi-smooth. If $E$ is a piecewise $\mathcal{C}^2$ function, then $E$ is strongly semi-smooth. It is known that proximal mappings of many interesting functions are piecewise linear or piecewise smooth.

\bibliography{manifold}
\bibliographystyle{plain}

\end{document}